\documentclass[11pt,reqno]{amsart}
\usepackage{amssymb,amsmath,hyperref}

\numberwithin{equation}{section}
\newtheorem{theorem}{Theorem}[section]
\newtheorem{proposition}[theorem]{Proposition}
\newtheorem{lemma}[theorem]{Lemma}

\newtheorem{remark}[theorem]{Remark}
\newcommand{\px}{\partial_x}

\begin{document}
\title[ fifth order KdV equation ]
  { On the fifth order KdV equation: local well-posedness and lack of uniform continuity of the solution map}
\author[Soonsik Kwon ]
{Soonsik Kwon}

\address{Soonsik Kwon \hfill\break
Department of Mathematics, University of California, Los Angeles, CA
90095-1555, USA} \email{rhex2@math.ucla.edu}

\date{}
\thanks{}
\thanks{} 

\begin{abstract}
In this paper we prove that the following fifth order equation arising from the KdV hierarchy
\begin{equation*}
\begin{cases}
\partial_tu + \px^5u + c_1\px u\px^2u + c_2u\px^3u = 0  \qquad  u : \mathbb{R}_t\times\mathbb{R}_x \rightarrow \mathbb{R}\\
u(0,x) =u_0(x)    \qquad \qquad u_0 \in H^s(\mathbb{R})
\end{cases}
\end{equation*}
is locally well-posed in $ H^s(\mathbb{R}) $ for $ s> \frac{5}{2}$. Also, we prove the solution map of the equation is not uniformly continuous.
\end{abstract}

\maketitle

\section{Introduction}

The Korteweg de Vries(KdV) equation has a fascinating property called \emph{complete integrability} in the sense that there is a Lax pair formulation of equations (or a bi-Hamiltonian structure). As is well known, this generates a hierarchy of Hamiltonian equations of order $ 2j+1$ and the corresponding Hamiltonians. Due to bi-Hamiltonian structure the flow of each equation conserves every Hamiltonian in the hierarchy. In particular, the KdV equation has infinitely many conservation laws and so does \eqref{fifthkdv}. The followings are first a few equations and their Hamiltonians with respect to one of two Hamiltonian structures.
\begin{align}
\partial_t u - \px u =0 , \qquad  \qquad & \int\frac{1}{2}u^2 \nonumber \\
\partial_t u - \px^3 u -6u\px u =0 , \qquad \qquad  & \int -\frac{1}{2}(\px u)^2 + u^3  \nonumber \\
\label{fifthkdv}\partial_t u - \px^5 -30u^2\px u + 20\px u\px^2 u + 10u\px^3u=0 , \qquad \qquad & \int\frac{1}{2}\px^2u^2-5u\px u^2 +\frac{5}{2}u^4  \\
          \vdots  \qquad \qquad &  \vdots \nonumber
\end{align}

In this paper, we consider the initial value problem of the fifth order equation \eqref{fifthkdv} in the hierarchy. Though the theory of complete integrability yields the global existence for Schwartz initial data and the soliton resolution phenomena, the well-posedness problem for low regularity initial data (or in the non-integrable case) is a very different problem, requiring the theory of dispersive PDE. The following equation generalizes \eqref{fifthkdv} to non-integrable case.
\begin{equation} \label{fifth}
\begin{cases}
\partial_tu + \px^5u + c_1\px u\px^2u + c_2u\px^3u = 0\\
u(0,x) =u_0(x)
\end{cases}
\end{equation}
where $ u : \mathbb{R}_t\times\mathbb{R}_x \rightarrow \mathbb{R} $ and $c_1,c_2$ are real constants. \\
We note that \eqref{fifth} also models several higher order water wave physics (see, for instance, \cite{benjamin}, \cite{benney}, \cite{olver}).\\
We consider the local well-posedness problem and a bad behavior of the flow map with the initial data in Sobolev space $H^s(\mathbb{R})$. Our first result is the local well-posedness for $ s>\frac{5}{2} $, as stated in the following theorem.
\begin{theorem}\label{lwp}
Let $s > \frac{5}{2} $. For any $u_0 \in H^s(\mathbb{R})$, there exists a time $ T \gtrsim \|u_0\|_{H^s}^{-\frac{10}{3}}$ and
a unique solution $u$ for the fifth order KdV equation $\eqref{fifth}$ satisfying
$$  u \in C([0,T], H^s(\mathbb{R})) \qquad \partial_x^3u \in L^1([0,T],L^\infty(\mathbb{R})). $$
Moreover, for any $R>0$, the solution map $u_0 \mapsto u(t) $ is continuous from the ball $ \{u_0 \in H^s(\mathbb{R}) :
\|u_0\|_{H^s} < R \}$  to $ C([0,T], H^s(\mathbb{R})) $. \\
\end{theorem}
Unlike the KdV equation, the local well-posedness problem cannot be solved by the contraction principle if we assume the initial data is in $H^s$. (In \cite{KPV94},\cite{KPV94-1} Kenig, Ponce and Vega proved the local well-posedness for a general dispersive equation with well-decaying initial data $ u_0 \in H^s(\mathbb{R})\cap L^2(|x|^mdx)$ for some large $s,m>0$ using the contraction principle.)\\
The following linear local smoothing is due to Kenig, Ponce and Vega \cite{KPV91}.
$$  \| \px^j e^{-t\px^{2j+1}}u_0(x) \|_{L^\infty_xL^2_t} \lesssim  \|u_0\|_{L^2_x}.   $$
One can observe the smoothing effect recovering \emph{two} derivatives ($j=2$) is not enough for the nonlinear term \emph{$u\px^3u$}. The fact that the nonlinear term has more derivatives than can be recovered by the smoothing effect causes a strong interaction between low and high frequencies data. This type of phenomenon is observed earlier in other dispersive equations, such as the Benjamin-Ono(BO) equation and the Kadomtsev-Petviashvili-I(KP-I) equation. In \cite{MST}, \cite{MST2} Molinet, Saut and Tzvetkov showed that the solution maps of these two equations are not $C^2 $ using examples localized in low and very high frequency. This already implies that Picard iteration is not available, since if it were the solution map would be real-analytic. Furthermore, in \cite{koch-tzvetkov2}, \cite{koch-tzvetkov3} Koch and Tzvetkov proved that the solution maps of these two equations are not uniformly continuous using the same low and high frequency nonlinear interaction. Our second result is an analog of theirs to the equation \eqref{fifth}.
\begin{theorem}\label{ill-posed}
Let $s>\frac{5}{2}$. Then, there exist constants $c$, $C$ and two sequences $(u_n)$
and $(v_n)$ of solution of the fifth order equation \eqref{fifth}  such that for every $t \in [0,1]$,
$$ \sup_n \|u_n(t,\cdot)\|_{H^s_x} + \sup_n
\|v_n(t,\cdot)\|_{H^s_x} \leq C, $$
$(u_n)$ and $(v_n)$
satisfy initially
$$ \lim_{n\rightarrow \infty} \|u_n(0,\cdot) - v_n(0,\cdot)
\|_{H^s_x} = 0, $$
But, for every $t \in [0,1]$,
$$ \liminf_{n\rightarrow \infty} \| u_n(t,\cdot) - v_n(t,\cdot) \|_{H^s_x}
\geq c |t|. $$
Moreover, if the equation \eqref{fifth} satisfies that for all $ t$,
\begin{itemize}
\item the $L^2$ conservation law in the sense that  $\|u(t)\|_{L^2_x} =\|u(0)\|_{L^2} $
\item an $H^3$ conservation law in the sense that
$$  \|u(t)\|_{H^3_x}  \lesssim \|u(0)\|_{H^3_x}  \qquad \text{for small}\quad \|u(0)\|_{H^3_x},   $$
\end{itemize}
then the same conclusion holds true for $ s>0$.\\
Note that the fifth order KdV equation \eqref{fifthkdv} satisfies these conditions. \\
\end{theorem}
In observance of the negative behavior of the solution map, one expects to prove the local well-posedness by the compactness method. Previously, in \cite{ponce94} Ponce proved the local well-posedness for Sobolev initial data $u_0 \in H^s(\mathbb{R}), s\geq 4$. Our result is an improvement of Ponce's. Our proof of the local well-posedness is based on the energy method combining with available global and local smoothing estimates. So, our basic strategy is the same as Ponce's.

One difficulty of \eqref{fifth}, which doesn't appear in the BO or the KP-I equation, is multi-derivatives in nonlinear terms, especially $u\px^3u$. The standard energy estimate gives only the following:

$$  \frac{d}{dt}\|\px^ku(t)\|^2_{L^2_x} \lesssim \|\px^3u\|_{L^\infty_x}\|\px^ku(t)\|_{L^2_x} + \Big|\int \px u \px^{k+1}u \px^{k+1}u\Big|   $$
because multi-derivatives may split when we take integrations by parts.

The last term is not favorable since it is a higher Sobolev norm than supposed to be. In Section~\ref{menergy} we modify the energy by adding a correctional term to cancel out the above last term. This idea is used several times in our analysis, actually whenever we take inner product to find the control of the time increment of a norm of a solution. In addition to this, we use the time chopping idea with respect to frequencies to improve the linear Strichartz estimate with a global smoothing effect. This idea was introduced by Koch and Tzvetkov \cite{koch-tzvetkov1} and improved by Kenig and Koenig \cite{kenig-koenig} in the context of the Benjamin-Ono equation. We also use Kato's local smoothing estimate complemented by the maximal function estimate.\\
In the proof of Thoerem~\ref{ill-posed}, we follow closely \cite{koch-tzvetkov2} for construction of the approximate solution. Unlike Benjamin-Ono equation, \eqref{fifth} may not have higher order conservation laws. Thus, we use well-posedness result for $s >\frac{5}{2}$, instead of conservation laws, for proving that the approximate solution is a good approximate solution in the $H^s$ sense. But in the region that well-posedness is unavailable ($ 0<s\leq \frac{5}{2}$) we have to assume that there is at least one higher conservation law and an $L^2$ conservation law, which \eqref{fifthkdv} satisfies. \\

The rest of the paper is organized as follows: in Section~\ref{notation} we summarize notations and standard lemmas for fractional derivatives.      In Section~\ref{menergy} we introduce the modified energy method to prove the energy estimate. In Section~\ref{linearestimate} we recall the linear estimates and show a refined Strichartz estimate. We also show Kato's local smoothing estimate for nonlinear solutions. In Section~\ref{lwpproof} we give the proof of Theorem~\ref{lwp}. After getting an \emph{a priori} bound, we use the $\epsilon$-approximate method introduced by Bona and Smith \cite{BS} for proving continuous dependence. Finally, in Section~\ref{illproof} we show by counter examples the lack of uniform continuity of the solution map.
\subsection*{Acknowledgements}
The author is very grateful to Terence Tao for many helpful conversations and encouragement. He is also indebted to him for suggesting this research problem. \\

\section{notation and preliminaries}\label{notation}
We use $X\lesssim Y$ when $X \leq CY $ for some $C$. We use $ X \sim  Y $ when $ X \lesssim Y $ and $ Y\lesssim X$.  Moreover, we use $ X \lesssim_s Y $ if the implicit constant depends on $s$, $C=C(s) $. \\
For a Schwartz function $u_0(x) $, we denote the linear solution $u(t,x)$ to the equation $ \partial_t u + \partial_x^5 u = 0 $ by
$$ u(t,x) = e^{-t\px^5}u_0(x) =  c \int \int e^{-t|\xi|^5}e^{i(x-y)\xi}u_0(y) dyd\xi. $$
Using this notation we have the Duhamel formula for the solution to the inhomogeneous linear equation $ \partial_t u + \partial_x^5 u + F = 0 $
$$ u(t,x) = e^{-t\px^5}u_0(x) - \int_0^t e^{-(t-t')\px^5}F(t',x) dt'.  $$
We use the space-time norm $L^p_TL^q_x $.
$$ \|u\|_{L^p_TL^q_x} = \Big( \int_0^T\Big( \int  |u(t,x)|^q dx \Big)^{p/q} dt\Big)^{1/p}  $$
with usual modification for $ p \emph{ or } q =\infty$.\\
We use the standard mollifier $ \rho_\epsilon $. More precisely, let $\rho$ be a function in $ C^\infty_0(\mathbb{R}) $ satisfying
$$ \rho \geq 0, \qquad \text{supp} \in [-1,1],\qquad \int \rho\, dx =1. $$
For $\epsilon >0 $, we denote $ \rho_\epsilon(x) := \frac{1}{\epsilon}\rho(\frac{x}{\epsilon}) $. For $ f \in L^1_{loc} $, denote $ f^\epsilon := \rho_\epsilon * f $. We use the following well-known lemma for mollified functions:

\begin{lemma}
Let $f \in W^{s,p}(\mathbb{R}^n)$ with $1 < p < \infty$ and $f^\epsilon = \rho_\epsilon * f $, then for any $\epsilon >0 $

\begin{equation}\label{v_epsilon}
\|f^\epsilon \|_{W^{s+t,p}} \lesssim_{s,t,p} \epsilon^{-t} \|f\|_{W^{s,p}} \qquad \text{ for } t>0,
\end{equation}
\begin{equation}\label{w_epsilon}
\|f-f^\epsilon \|_{W^{s-t,p}} \lesssim_{s,t,p} \epsilon^{t} \|f\|_{W^{s,p}} \qquad \text{ for } 0 \leq t \leq s.
\end{equation}
\end{lemma}

We denote $ D^s$ is a homogeneous fractional derivative whose symbol is $ |\xi|^s $, while $J^s$ is an inhomogeneous derivative whose symbol is $ (1+\xi^2)^{\frac{s}{2}} $. Note $ \|J^su\|_{L^2} = \|u\|_{H^s} $. We use the following standard fractional Leibnitz rule and the Kato-Ponce commutator estimate. 

\begin{lemma}
(a)Let $ 0 \leq s,s_1,s_2<1$ and $ 1<p,p_1,p_2<\infty $ be such that $\frac{1}{p_1} +\frac{1}{p_2} = \frac{1}{p}$, and $s_1+s_2=s $. Then
\begin{equation}\label{product}
\| D^s(fg)-(D^sf)g-f(D^sg) \|_{L^p} \lesssim_{p,p_1,p_2,s,s_1,s_2} \|D^{s_1}f\|_{L^{p_1}}\|D^{s_2}g\|_{L^{p_2}}.
\end{equation}
Moreover, if $s_2=0$, $ p_2= \infty $ is allowed.\\
(b) In particular, for $ s \geq 1$
\begin{equation}\label{K-P}
\| [D^s;f]g \|_{L^p} \lesssim_{s,p} \|\nabla f\|_{L^{p_1}}\|D^{s-1}g\|_{L^{p_2}} + \|D^sf\cdot g\|_{L^{p}}
\end{equation}
where $p_1=\infty $ is allowed.\\
\end{lemma} 
For the proof of this lemma, we refer, for instance, \cite{kato-ponce}, \cite{KPV93}.\\
In addition, in our analysis for the modified energy estimate we need a generalized form of the commutator estimate.

\begin{lemma}[Generalized commutator estimate]\label{commutator}
Let $s>0$. Then
 \begin{align*} \|D^s(u\px^3v) - uD^s(\px^3v) - s
\px uD^s(\px^2v) - \frac{s(s-1)}{2}&\px^2uD^s(\px v) \|_{L^2} \\ \lesssim_s &
\|\px^3u\|_{L^\infty}\|D^sv\|_{L^2}  + \|\px^3v\|_{L^\infty}\|D^su\|_{L^2},\\
\|D^s(\px u\px^2v) - \px uD^s(\px^2v) - s\px^2uD^s(\px v) \|_{L^2}
\lesssim_s &\|\px^3u\|_{L^\infty}\|D^sv\|_{L^2} + \|\px^3v\|_{L^\infty}\|D^su\|_{L^2}
\end{align*} hold true.
\end{lemma}
\begin{proof}
This proof is similar to that of the Kato-Ponce commutator estimate, an application of the Coifman-Meyer theorem (see \cite{kato-ponce},\cite{coifman-meyer}).
Here we prove only the first one. The second one is proved in the same way.
Let $\sigma(\xi_1,\xi_2)$ be the symbol of the pseudo-differential operator. More precisely,
$$ \sigma(\xi_1,\xi_2) = i^3 \Big\{ |\xi_1 + \xi_2|^s\xi_2^3 - |\xi_2|^s\xi_2^3 -s\xi_1|\xi_2|^s\xi_2^2 - \frac{s(s-1)}{2}\xi_1^2|\xi_2|^s\xi_2  \Big\} $$
\begin{align*}
D^s(u\px^3v) - uD^s(\px^3v) - s
\px uD^s(\px^2v) - \frac{s(s-1)}{2}\px^2uD^s(\px v)(x) \\
= c\int\int e^{ix(\xi_1+\xi_2)} \sigma(\xi_1,\xi_2)\widehat{u}(\xi_1)\widehat{v}(\xi_2) d\xi_1d\xi_2
\end{align*}
As usual, we decompose $\sigma$ into low-high, high-low, high-high paraproduct
$$ \sigma(\xi_1,\xi_2) = \sigma_{lh}(\xi_1,\xi_2) + \sigma_{hh}(\xi_1,\xi_2) +\sigma_{hl}(\xi_1,\xi_2) $$
\begin{align*}
\sigma_{k}(\xi_1,\xi_2) = \sigma(\xi_1,\xi_2) \cdot \pi_{k}\left(\frac{\xi_1}{\xi_2}\right)  \quad & \text{supp}\,\pi_{lh} \subset [-1/3,1/3]\\
& \text{supp}\,\pi_{hh} \subset [1/4,4]\cup[-4,-1/4]\\
& \text{supp}\,\pi_{hl} \subset (-\infty,-3]\cup[3,\infty)\\
  \pi_{k} \geq 0 \qquad & \pi_{lh} + \pi_{hh} + \pi_{hl} = 1 \qquad k=lh,hh,hl
\end{align*}
We are going to show the inequality for each case. To prove $ \sigma_{lh}$ and $ \sigma_{hl}$ we need to use the following lemma due to Coifman and Meyer.
\begin{lemma}[Coifman-Meyer \cite{coifman-meyer}]\label{coifman-meyer}
Let $\sigma \in C^\infty(\mathbb{R}^m\times\mathbb{R}^m -\{(0,0)\})$ satisfy
\begin{equation}\label{cm-condition}
|\partial^{\alpha}_{\xi_1}\partial^{\beta}_{\xi_2}\sigma(\xi_1,\xi_2) | \lesssim_{\alpha,\beta} \Big(|\xi_1+\xi_2| \Big)^{-|\alpha|-|\beta|}
\end{equation}
for $(\xi_1,\xi_2)\neq (0,0) $ and any $ \alpha,\beta \in (\mathbb{Z}^+)^m$. If $\sigma(D) $ denotes the bilinear operator
$$ \sigma(D)(u,v)(x) = \int\int e^{ix\cdot (\xi_1+\xi_2)} \sigma(\xi_1,\xi_2)\widehat{u}(\xi_1)\widehat{v}(\xi_2) d\xi_1d\xi_2 $$
then
$$ \|\sigma(D)(u,v)\|_{L^2} \lesssim \|u\|_{L^\infty}\|v\|_{L^2} $$

\end{lemma}

At first, we consider the low-high paraproduct. In this case we need to use cancellation of terms in the symbol $ \sigma_{lh} $.
Since $\xi_1,\xi_2$ are restricted to the region $ |\xi_1| <\frac{1}{3}|\xi_2| $, one can rewrite the first term of $\sigma_{lh}$ as 

\begin{align*}
 |\xi_1+\xi_2|^s \xi_2^3\frac{1}{\xi_1^3|\xi_2|^s} &= \xi_2^3|\xi_2|^s\left(1+\frac{\xi_1}{\xi_2}\right)^s\frac{1}{\xi_1^3|\xi_2|^s} \\
  & = \frac{\xi_2^3}{\xi_1^3}\Big(1+ \frac{\xi_1}{\xi_2} + s\left(\frac{\xi_1}{\xi_2}\right)^2 +\frac{s(s-1)}{2}\left(\frac{\xi_1}{\xi_2}\right)^3 + \cdots \\
  &   + { s\choose k}\left(\frac{\xi_1}{\xi_2}\right)^k + \cdots    \Big)\\
\end{align*}
\begin{align*}
\sigma_{lh}\left(\xi_1,\xi_2\right) \frac{1}{(i\xi_1)^3|\xi_2|^s} &= \frac{i^3\xi_2^3}{i^3\xi_1^3} \Big\{{s\choose 3}\left(\frac{\xi_1}{\xi_2}\right)^3 + {s\choose 4}\left(\frac{\xi_1}{\xi_2}\right)^4 + \cdots \Big\} \\
& = { s\choose 3} + {s\choose 4}\left(\frac{\xi_1}{\xi_2}\right) + {s\choose 5}\left(\frac{\xi_1}{\xi_2}\right)^2 + \cdots
\end{align*}
Since $ |\xi_1| < 1/3|\xi_2| $, this infinite series converges absolutely and so are their partial derivatives. Furthermore, one can easily check that $ \sigma_{lh}(\xi_1,\xi_2)\frac{1}{(i\xi_1)^3|\xi_2|^s} $ satisfies the condition $\eqref{cm-condition}$. Thus, by the Coifman-Meyer lemma

\begin{equation*}
\| \int\int e^{ix(\xi_1+\xi_2)}\sigma_{lh}(\xi_1,\xi_2)\frac{1}{(i\xi_1)^3|\xi_2|^s}\widehat{\px^3u}(\xi_1)\widehat{D^sv}(\xi_2) d\xi_1d\xi_2\|_{L^2}
 \lesssim \|\px^3 u\|_{L^\infty}\|D^sv\|_{L^2}.
\end{equation*}

Next, we turn to the high-low paraproduct $\sigma_{hl} $ (i.e. $|\xi_1| > 3|\xi_2|)$ . In this case we estimate each of four terms in $\sigma_{hl}$ separately. The first term in $\sigma_{hl}$ (multiplying by $\frac{1}{(i\xi_2)^3|\xi_1|^s}$) can be written as 
$$
 \frac{i^3\xi_2^3|\xi_1+\xi_2|^s}{i^3\xi_2^3|\xi_1|^s}\pi_{hl}\left(\frac{\xi_1}{\xi_2}\right) = \frac{\xi_2^3|\xi_1|^s}{\xi_2^3|\xi_1|^s}\Big(1+ \frac{\xi_2}{\xi_1} \Big)^s\pi_{hl}\left(\frac{\xi_1}{\xi_2}\right) =  \Big(1+ \frac{\xi_2}{\xi_1} \Big)^s\pi_{hl}\left(\frac{\xi_1}{\xi_2}\right)
$$
which is smooth and satisfies the condition $\eqref{cm-condition}$. Thus, by the Coifman-Meyer lemma
$$
\| \int\int e^{ix(\xi_1+\xi_2)}\frac{i^3\xi_2^3|\xi_1+\xi_2|^s}{i^3\xi_2^3|\xi_1|^s}\pi_{hl}\left(\frac{\xi_1}{\xi_2}\right)\widehat{D^su}(\xi_1)\widehat{\px^3v}(\xi_2) d\xi_1d\xi_2\|_{L^2}
 \lesssim \|\px^3 v\|_{L^\infty}\|D^su\|_{L^2}.
$$
For other symbol terms $ i^3 \xi_1^k|\xi_2|^s\xi_2^{3-k} , k=0,1,2 $, we divide by $ (i\xi_1)^3|\xi_2|^s $ to obtain
$$   \frac{i^3 \xi_1^k|\xi_2|^s\xi_2^{3-k} }{(i\xi_1^3)^3|\xi_2|^s}\pi_{hl}(\xi_1/\xi_2) = \left(\frac{\xi_2}{\xi_1}\right)^{3-k}\pi_{hl}(\xi_1/\xi_2)$$
for $ k=0,1,2$. Since these symbols satisfy $\eqref{cm-condition}$, we get
$$  \| \int\int e^{ix(\xi_1+\xi_2)}\frac{i^3\xi_1^k|\xi_2|^s\xi_2^k}{i^3\xi_1^3|\xi_2|^s}\pi_{hl}\left(\frac{\xi_1}{\xi_2}\right)\widehat{\px^3u}(\xi_1)\widehat{D^sv}(\xi_2) d\xi_1d\xi_2\|_{L^2}
 \lesssim \|\px^3 u\|_{L^\infty}\|D^sv\|_{L^2}. $$

In the high-high paraproduct case, $ |\xi_1+\xi_2|^s\xi_2^3$ has singularities on $\xi_1 = -\xi_2$ line. So, it's not a smooth Coifman-Meyer multiplier. But this case can be proved using Littlewood-Paley operators.
Let $P_N$ be a Littlewood-Paley projection into frequency $\xi \sim N$ where N is a dyadic number.
It suffices to estimate
$$
\|D^s(u\px^3v)\|_{L^2} \sim \|\sum_N \sum_{N/4 \leq M \leq 4N} D^s(P_NuP_M\px^3v) \|_{L^2}. $$
Set $ \widetilde{P}_N := \sum_{N/4 \leq M \leq 4M} P_M $.
Using the Littlewood-Paley inequality and $ \sum_N \widetilde{P}_N = 5 $
\begin{align*}
\|D^s(u\px^3v)\pi_{hh}\|_{L^2}   & \sim \|\sum_N D^s(P_NuP_N\px^3v)  \|_{L^2} \\
    & \lesssim \sum_j \|\sum_N \widetilde{P}_{2^{-j}N}[ D^s(\widetilde{P}_Nu\widetilde{P}_N\px^3v)]\|_{L^2}\\
    & \lesssim \sum_j \Big(\sum_N \|\widetilde{P}_{2^{-j}N}[ D^s(\widetilde{P}_Nu\widetilde{P}_N\px^3v)] \|_{L^2}^2  \Big)^{1/2}\\
    & \lesssim_s \sum_j \Big(\sum_N \| (2^{-j}N)^s \widetilde{P}_Nu\widetilde{P}_N\px^3v) \|_{L^2}^2  \Big)^{1/2}\\
    & \lesssim \sum_j 2^{-js} \Big(\sum_N N^s\|\widetilde{P}_Nu\|^2_{L^2}\|\px^3v\|^2_{L^\infty}  \Big)^{1/2}\\
    & \lesssim_s \|\px^3v\|_{L^\infty}\|D^su\|_{L^2}.
\end{align*}

The other terms in the high-high paraproduct, $ \left(\frac{\xi_2}{\xi_1}\right)^{3-k}\pi_{hh}(\xi_1/\xi_2)$, are Coifman-Meyer multipliers and can be estimated as high-low case.

\end{proof}

\section{modified energy}\label{menergy}
In this section, we prove an energy estimate for local solutions. We want to control the time increment of $ \|D^su(t)\|_{L^2_x} $ using itself and other norms of the same size. But due to multi-derivatives in the nonlinear terms (especially, $u\px^3u$) the standard energy method (combining with the commutator estimate in the fractional derivative case) gives only
$$  \frac{d}{dt}\|D^su(t)\|^2_{L^2_x} \lesssim \|\px^3u\|_{L^\infty_x}\|D^su(t)\|_{L^2_x} + \left|\int \px u D^s\px u D^s\px u \right|.  $$
The last term of which is not favorable.
We use some modification of the energy in order to cancel out the last term.

\begin{proposition}\label{energy}
Let $s\geq 1$ and $u(t,x)$ be a Schwartz solution to the equation \eqref{fifth}. Define the modified energy
$$ E_s(t):= \|D^su(t)\|_{L^2}^2 + \|u(t)\|_{L^2}^2 + a_s\int
u(t)D^{s-2}\partial_xu(t)D^{s-2}\partial_xu(t). $$
Then, there exist constants $a_s$, $C$ and $C'$ so that if $\|u(t)\|_{H^s_x} \leq \min(\frac{1}{2},\frac{1}{2a_sC})$, then
\begin{gather}
\label{energy31}      \frac{1}{2}\|u\|^2_{H^s}  \leq   E_s \leq  \frac{3}{2}\|u\|^2_{H^s} \\
\frac{d}{dt}E_s(t)  \lesssim_s  \|\partial_x^3u\|_{L^\infty} E_s(t)   \label{energy1} 
\end{gather}
and so if $\|u\|_{L^\infty_TH^s_x} \leq \min(\frac{1}{2},\frac{1}{2a_s})$, then
\begin{equation}\label{energybound}
\sup_{[0,t]} \|J^su\|_{L^2_x}  \leq  3\, e^{C'\int^t_0 \|\px^3u(t')\|_{L^\infty_x}dt' } \|J^su(0)\|_{L^2_x}.
\end{equation}

\end{proposition}

\begin{proof}
  Using Holder inequality and Sobolev embedding the third term of $E_s$ is bounded by $ C\cdot a_s\|u\|_{H^s}^3 $
We have
\begin{equation*}
\frac{1}{2}\|u(t)\|_{H^s_x}^2  \leq \, E_s(t) \, \leq \frac{3}{2}\|u(t)\|_{H^s_x}^2.
 \end{equation*}
To prove \eqref{energy1} one we take $ D^s\partial_x^2$ derivative on the equation $\eqref{fifth} $ and integrate against $D^s\partial_x^2u$.
\begin{equation*}
 \frac{1}{2}\frac{d}{dt}\|D^su\|^2 + c_1 \int D^s(u\partial_x^3u)D^su
  + c_2\int D^s(\partial_xu\partial_x^2u)D^su = 0.
\end{equation*}
To use the commutator estimate we add and subtract terms.

\begin{align*}
\frac{1}{2} & \frac{d}{dt}\|D^su\|^2 \leq \\
 & c_1\int \Big[ D^s (u\px^3u) - uD^{s-2}\px^5u - s\px uD^{s-2}\px^4u -\frac{s(s-1)}{2} \px^2uD^{s-2}\px^3u \Big]D^su \\
 &+ c_2\int \Big[ D^s(\px u\px^2u) - \px uD^{s-2}(\px^4u) - s\px^2uD^s(\px^3 u) \Big]D^su\\
 &+  \int \Big[ c_3 uD^{s-2}\px^5u + c_4 \px uD^{s-2}\px^4u  + c_5 \px^2uD^{s-2}\px^3u \Big]D^su\\
 & =: \text{I} + \text{II} + \text{III}.
\end{align*}
By Lemma~\ref{commutator}
\begin{equation*} \text{I} + \text{II} \lesssim \|\px^3u\|_{L^\infty}\|D^su\|^2_{L^2}.  \end{equation*}
After some integrations by parts we have
\begin{align*}
\text{III}  &=  \int \Big[ c_3 uD^{s-2}\px^5u + c_4 \px uD^{s-2}\px^4u  + c_5 \px^2uD^{s-2}\px^3u \Big]D^su \\
     &= d_1 \int \px^3uD^s u D^{s-2} \px^2 u + d_2\int \px uD^{s-2}\px^3 u D^{s-2} \px^3 u.
\end{align*}
On the other hand,
\begin{align*}
\frac{d}{dt} &  \int uD^{s-2}\partial_xuD^{s-2}\partial_xu \\
   &= \int u_t D^{s-2}\partial_xuD^{s-2}\partial_xu +
\int uD^{s-2}\partial_xu_tD^{s-2}\partial_xu\\
 &= -\int\partial_x^5uD^{s-2}\partial_xuD^{s-2}\partial_x u  -2\int uD^{s-2}\px^6 u D^{s-2}\px u  \\
 & -c_1\int u\px^3uD^{s-2}\px uD^{s-2}\px u + 2uD^{s-2}[\px u\px^3 u]D^{s-2}\px u + 2uD^{s-2}[u\px^4u]D^{s-2}\px u \\
 & -c_2\int \px u\px^2uD^{s-2}\px uD^{s-2}\px u + 2uD^{s-2}[\px^2u\px^2u]D^{s-2}\px u + 2uD^{s-2}[\px u\px^3u]D^{s-2}\px u\\
 &=:  A + B
\end{align*}
where $ A $ is a linear combination of terms of degree of $3$ (there are three $u$'s) and $B$ is a linear combination of terms of degree $4$. We use integrations by parts to change $A$ to a linear combination of the following three terms.
\begin{align*}
A & = -\int\partial_x^5uD^{s-2}\partial_xuD^{s-2}\partial_x u  -2\int uD^{s-2}\px^6 u D^{s-2}\px u \\
  & = \alpha_1 \int \px^5uD^{s-2}\px u D^{s-2} \px u + \alpha_2 \int \px^3uD^s u D^{s-2} \px^2 u + \alpha_3 \int \px^1uD^{s-2}\px^3 u D^{s-2} \px^3 u.
\end{align*}
But by direct computation one can easily check $ \alpha_1 =0$. One can choose $a_s$  so that $ a_s\cdot\alpha_3 + d_2 = 0 $. Then,
$$ \text{III} + A = (d_1 + a_s\alpha_2) \int \px^3uD^s u D^{s-2} \px^2 u \lesssim \|\px^3u\|_{L^\infty_x}\|D^su\|_{L^2_x}.$$
Now we estimate $B$. First of all, one can easily observe (using Sobolev embedding and $\|u(t)\|_{H^s_x} \leq \frac{1}{2}$)
\begin{equation*}
\left|\int u\px^3uD^{s-2}\px uD^{s-2}\px u\right|  + \left|\int \px u\px^2uD^{s-2}\px uD^{s-2}\px u\right| \lesssim \|\px^3u\|_{L^\infty_x}^2\|J^su\|_{L^2_x}^2.
\end{equation*}
To estimate other terms we use the commutator estimate \eqref{K-P}.
\begin{align*}
\int uD^{s-2}[\px u\px^3 u]D^{s-2}\px u = \int \Big[D^{s-2}(\px u\px^3 u) - \px uD^{s-2}\px^3u \Big]D^{s-2}\px u \\
+ u\px u D^{s-2}\px^3 uD^{s-2}\px u.
\end{align*}
From the commutator estimate and integrations by parts these terms are bounded by
$$  O\left(\|\px^3u\|_{L^\infty_x}\|J^su\|_{L^2_x}^2\right).  $$
Other terms' bounds follow in the same way.
Getting everything together and using Gronwall's inequality we have shown
\begin{gather*}
\frac{d}{dt}E_s(t) \lesssim_s \|\px^3u(t)\|_{L^\infty_x}E_s(t), \\
E_s(t) \leq e^{\int^t_0 \|\px^3u(t')\|_{L^\infty_x}dt'} E_s(0). \\
\end{gather*}
By \eqref{energy31}, \eqref{energybound} follows.
\end{proof}

\section{linear estimate and local smoothing}\label{linearestimate}
In this section we provide the linear estimates and a local smoothing for nonlinear solutions.

\begin{lemma}[Strichartz estimate, \cite{KPV93}]\label{strichartz}
Let $u_0 \in L^2_x$.
\begin{equation}
\| D^\alpha e^{-t\partial_x^5} u_0 \|_{L^q_tL^r_x} \lesssim \|u_0\|_{L^2_x}
\end{equation}

for $-\alpha + \frac{5}{q} + \frac{1}{r} = \frac{1}{2} $, $ 0 \leq \alpha \leq 3/q $ and $ 2 \leq q,r \leq \infty$.
\end{lemma}
The following proposition is a refined version of the Strichartz estimate, introduced by Koch and Tzvetkov \cite{koch-tzvetkov1} and improved by Kenig and Koenig \cite{kenig-koenig} in the context of the Benjamin-Ono equation.

\begin{proposition}[Refined Strichartz estimate, \cite{kenig-koenig}]\label{refinedstrichartz}
Let $ T \leq 0$ and $ 0 \leq \alpha  $.
Let $u$ be a Schwartz solution to the linear fifth order equation $ \partial_tu + \partial_x^5u + F = 0$.
For any $\epsilon > 0$,
\begin{equation}\label{restri}
\| \px^3u\|_{L^2_TL^\infty_x} \lesssim_\alpha \|D^{3-\frac{3}{4}+\frac{\alpha}{4}+\epsilon}u \|_{L^\infty_TL^2_x} + \|D^{3-\frac{3}{4}-\frac{3\alpha}{4}+\epsilon}F \|_{L^2_TL^2_x} + \|u\|_{L^2_TL^2_x} + \|F\|_{L^2_TL^2_x}
\end{equation}

\end{proposition}

\begin{remark}
In our analysis, the optimal choice is $\alpha=1 $. With this \eqref{restri} gives
\begin{equation}\label{restri3}
 \|\partial_x^3u\|_{L^1_TL^\infty_x} \lesssim \|J^s u\|_{L^\infty_TL^2_x} + \|J^{s-1}F\|_{L^2_TL^2_x}
\end{equation}
for $s>\frac{5}{2}$, which determines the regularity threshold of our analysis.

\end{remark}

\begin{proof}
 The proof is very similar to that of the Benjamin-Ono equation given in \cite{kenig-koenig}. Unlike the BO equation, from the global smoothing effect of the Strichartz estimate we gain $\frac{3}{4}$ derivatives. We provide the proof for the reader's convenience.\\
Let $P_N$ be a Littlewood-Paley operator where N is a dyadic number and denote $ P_Nu=u_N$. By Sobolev embedding and the Littlewood-Paley inequality
\begin{align*}
\|\px^3u\|_{L^2_TL^\infty_x} & \lesssim_{\epsilon',r} \|J^{\epsilon '}\px^3u\|_{L^2_TL^r_x} \\
           & \sim \|\big( \sum_N|J^{\epsilon '}\px^3u_N|^2 \big)^{1/2} \|_{L^2_TL^r_x} \\
           & \leq \Big(\sum_N \|J^{\epsilon '}\px^3u_N\|^2_{L^2_TL^r_x} \Big)^{1/2}
\end{align*}
where $\epsilon'$  and $ r ( > 1/\epsilon') $ to be chosen later.\\
It suffices to show for $r>2$
\begin{equation}\label{restri2}
\|\px^3u_N\|_{L^2_TL^r_x} \lesssim \|D^{3-\frac{3}{4} +\frac{\alpha}{4}-(\alpha-3)\frac{1}{2r} }u_N\|_{L^\infty_TL^2_x} +
       \|D^{3-\frac{3}{4} -\frac{3\alpha}{4} -(\alpha-3)\frac{1}{2r}}F_N\|_{L^2_TL^2_x}
\end{equation}

for $ N= 2^k , k \geq 1 $. (The case $k=0$ is handled by $\|u\|_{L^\infty_TL^2_x} + \|F\|_{L^2_TL^2_x}$ by the Strichartz estimate and Bernstein's inequality).\\

Now, we chop out the time interval into subintervals of unit length. Let $[0,T]= \cup_j I_j, I_j=[a_j,b_j]$ with $ |I_j| \sim N^{-\alpha} $.
Then the number of j's is $O(TN^\alpha)$ \\
Let $q $ be so that $-\frac{3}{q} + \frac{5}{q} + \frac{1}{r} = \frac{1}{2}$. Using Lemma \ref{strichartz} and $\partial_tu_N + \px^5u_N + F_N = 0 $
\begin{align*}
&\|\px^3u_N\| _{L^2_TL^r_x} \\
& = \Big(\sum_j \|\px^3u_N\|^2_{L^2_{I_j}L^r_x}  \Big)^{1/2}  \\
      & \leq N^{-\alpha(\frac{1}{2} -\frac{1}{q})} \Big( \sum_j \|\px^3u_N\|^2_{L^2_{I_j}L^r_x} \Big)^{1/2} \qquad (\because |I_j| \sim  N^{-\alpha})\\
      & \lesssim  N^{-\alpha(\frac{1}{2} -\frac{1}{q})} \sum_j \Big(\| e^{-(t-a_j)\partial^5_x}\px^3u_N(a_j)\|^2_{L^q_{I_j}L^r_x} +
        \| \int^t_{a_j} e^{-(t-t')\px^5}\px^3F_N(t')dt' \|^2_{L^q_{I_j}L^r_x}\Big)^{1/2} \\
      & \lesssim  N^{-\alpha(\frac{1}{2} -\frac{1}{q})} \Big\{ \Big(\sum_j \|D^{-\frac{3}{q}}\px^3u_N\|^2_{L^\infty_TL^2_x} \Big)^{1/2}
      + \Big( \sum_j \Big( \int_{I_j} \|D^{-\frac{3}{q}}\px^3F_N\|_{L^2_x} dt \Big)^2 \Big)^{1/2}   \Big\} \\
      & \lesssim  N^{-\alpha(\frac{1}{2} -\frac{1}{q})} \Big\{ \Big(\sum_j \|D^{-\frac{3}{q}}\px^3u_N\|^2_{L^\infty_TL^2_x} \Big)^{1/2}
      + \Big( \sum_j N^{-\alpha} \int_{I_j} \|D^{-\frac{3}{q}}\px^3F_N\|^2_{L^2_x} dt \Big)^{1/2}   \Big\} \\
      & \lesssim  N^{-\alpha(\frac{1}{2} -\frac{1}{q})}N^{\frac{\alpha}{2}}\| D^{3-\frac{3}{q}}u_N\|_{L^\infty_TL^2_x} +
          N^{-\alpha+\frac{\alpha}{q}} \Big(\int^T_0 \|D^{3-\frac{3}{q}} F_N\|^2_{L^2_x} dt \Big)^{1/2} \\
      & \lesssim  \|D^{3-\frac{3}{q} +\frac{\alpha}{q}}u_N\|_{L^\infty_TL^2_x} +
       \|D^{3-\frac{3}{q} -\alpha +\frac{\alpha}{q}}u_N\|_{L^2_TL^2_x}\\
      & \lesssim  \|D^{3-\frac{3}{4} +\frac{\alpha}{4}-(\alpha-3)\frac{1}{2r} }u_N\|_{L^\infty_TL^2_x} +
       \|D^{3-\frac{3}{4} -\frac{3\alpha}{4} -(\alpha-3)\frac{1}{2r}}F_N\|_{L^2_TL^2_x}.
\end{align*}
At the last step we used $ \frac{1}{q} =\frac{1}{4} -\frac{1}{2r} $. Once we get $\eqref{restri2}$, for given $\epsilon>0$, $\eqref{restri}$ follows by choosing $\epsilon'$ and $r $ so that $ \epsilon'-\frac{\alpha-3}{2r} < \epsilon$.

\end{proof}

Next, we state the maximal estimate for linear solutions proved by Kenig, Ponce and Vega \cite{KPV91}.

\begin{lemma}[Maximal estimate]\label{maximal}
Assume $u_0 \in H^{5/4+\eta} $ for some $\eta > 0$.
Then
$$ \| e^{-t\px^5}u_0 \|_{L^2_xL^\infty_T} \lesssim \Big(\sum_j \|e^{-t\px^5} u_0 \|^2_{L^\infty([0,T]\times[j,j+1))}  \Big)^{1/2}
  \lesssim_{T,\eta} \|u_0\|_{H^{5/4+\eta}}. $$
\end{lemma}

Next, for solutions to the nonlinear equation \eqref{fifth} we prove the local smoothing estimate. This is an analog of Kato's local smoothing for the KdV equation. This approach to \eqref{fifth} was done by Ponce \cite{ponce94}, here we improve Ponce's result by adding the modified energy idea used in the proof of Proposition~\ref{energy}.

\begin{lemma}[Local smoothing estimate]\label{localsmoothing}
Let $s>\frac{5}{2}$. Let $I$ be an interval of unit length. If $u$ is a Schwartz solution to the fifth order KdV $\eqref{fifth} $ , then
\begin{equation}\label{localsmoothingineq}
\Big( \int^T_0\,\int_I |D^{s-2}\px^4u|^2 dxdt \Big) \lesssim_{|I|,s} \big(1+ \|\px^3u\|_{L^1_TL^\infty_x}+\|J^su\|^2_{L^\infty_TL^2_x} \big)\|J^su\|^2_{L^\infty_TL^2_x}.
\end{equation}
\\
\end{lemma}

\begin{proof}
Let $ \phi \in C^\infty(\mathbb{R}) $ be an increasing function whose derivative $ \phi' \in C^\infty_0,$ $\phi'\geq 0$ and $\phi'\equiv 1 $ on $I$. Our local smoothing estimate recover two derivatives. We take two steps. First we prove a local smoothing estimate recovering one derivative.  The claim is the following:
\begin{equation}\label{1localsmoothing}
 \int^T_0\,\int |D^{s-2}\px^3u|^2 \phi' dxdt  \lesssim_\phi \big(1+ \|J^3u\|^2_{L^1_TL^\infty_x} + \|J^su\|_{L^\infty_TL^2_x}\big)\sup_{[0,T]}\|J^su\|^2_{L^2_x}.\\
\end{equation}
From \eqref{fifth}
\begin{align*}
\frac{1}{2}& \frac{d}{dt}\int D^{s-2}\px u D^{s-2}\px u \phi = \int D^{s-2}\px\partial_t uD^{s-2}\px u \phi \\
   &= \int D^{s-2}\px^6 D^{s-2}\px u\phi + c_1\int D^{s-2}\px(u\px^3 u) D^{s-2} \px u \phi + c_2\int D^{s-2}\px(\px u\px^2 u) D^{s-2}\px u\phi \\
   &=: A+B+C.
\end{align*}
A few integrations by parts show
$$ A = c\int (D^{s-2}\px^3 u)^2\phi' + c\int(D^s u)^2 \phi^{(3)} + c\int(D^{s-2}\px u)^2 \phi^{(5)} $$
where the last two terms are $O\left(\|J^s\|_{L^2_x}^2\|\px^3 u\|_{L^\infty_x}\right) $
To estimate $ B$ and $C$ we use Lemma $\ref{commutator}$ and integrations by parts.
\begin{align*}
B &= c\int D^{s-2}(u\px^3 u)D^{s-2}\px u \phi + c\int D^{s-2}(u\px^3 u)D^{s-2}\px u \phi' \\
  &= O(\|J^s u\|_{L^2_x}^2\|\px^3 u\|_{L^\infty_x}) + c\int uD^{s-2}\px^3 u D^s u \phi,
\end{align*}
\begin{align*}
\int uD^{s-2}\px^3 u D^s u \phi  &= c\int \px u(D^s u)^2 \phi + c\int u(D^s u)^2 \phi' \\
    & = O(\|J^s u\|_{L^2_x}^3),
\end{align*} where we used Sobolev embedding.\\
Similarly, one can prove that $C$ is also $O\left(\|J^s\|_{L^2_x}^2(\|J^3 u\|_{L^\infty_x}+ \|J^su\|_{L^2_x})\right)$.
Then, $\eqref{1localsmoothing}$ is obtained by integrating in time.\\
Now, let's prove full local smoothing estimate $\eqref{localsmoothingineq}$. For this proof we use a cancellation as used in the proof of Proposition $\ref{energy}$.
Define
\begin{align*}
 E_\phi(t) &= \int (D^s u)^2 \phi + a_s \int u (D^{s-2}\px u)^2 \phi \\
       & = E_1(t) +E_2(t).
\end{align*}
Using $\eqref{fifth}$ and integrations by parts
\begin{align*}
\frac{d}{dt}E_1 &= c\int (D^{s-2}\px^4 u)^2 \phi' + c\int (D^{s-2}\px^3)^2 \phi^{(3)} +c\int (D^s u)^2 \phi^{(5)} +\\
    & c\int D^s(u\px^3 u)D^s u \phi + c\int D^s(\px u\px^2 u)D^s u \phi. \\
\end{align*}
Here, each constant $c$ has a different value. From $\eqref{1localsmoothing}$ the second and third term are $ O\left((1+ \|\px^3u\|_{L^1_TL^\infty_x}+\|J^su\|_{L^\infty_TL^2_x} \big)\|J^su\|^2_{L^\infty_TL^2_x}\right) $ after integration in time. When we estimate the last two terms, we use Lemma $\ref{commutator}$. Since the proof of the two are similar we give a proof of the fourth term.
\begin{align*}
\int D^s &(u\px^3 u)D^s u \phi = \\
 & \int\Big[D^s(u\px^3u) - uD^{s-2}\px^5 u - s \px uD^{s-2}\px^4 u - \frac{s(s-1)}{2}\px^2 u D^{s-2}\px^3 u     \Big]D^s u \phi \\
  &+ \int uD^{s-2}\px^5 uD^s u \phi   +s\int \px uD^{s-2}\px^4 uD^s u \phi + \frac{s(s-1)}{2} \int \px^2 u D^{s-2}\px^3 uD^s u \phi.
\end{align*}
The first term is done by Lemma~$\ref{commutator}$. Several integrations by parts show that the rest are a linear combination of
\begin{align*}
\int \px u (D^{s-2}\px^3 u)^2 \phi, \, \int u(D^{s-2}\px^3 u)^2\phi', \,\int \px^3 u (D^s u)^2 \phi \\
\int \px^2 u (D^s u)^2 \phi', \,   \int \px u (D^s u)^2 \phi'', \,   \int u (D^s u)^2 \phi'''.
\end{align*}
Five terms \emph{except the first term} are $ O\left(\big(1+ \|\px^3u\|_{L^1_TL^\infty_x}+\|J^su\|_{L^\infty_TL^2_x} \big)\|J^su\|^2_{L^\infty_TL^2_x}\right) $ after integration in time. (for the second one, we used $\eqref{1localsmoothing}$.) So, we need $E_2$ to cancel out the first term $\int \px u (D^{s-2}\px^3 u)^2 \phi $. Since the method used here is very similar to the proof of Proposition $\ref{energy}$, a sketch is enough.
\begin{align*}
\frac{d}{dt}E_2(t) &= 2\int uD^{s-2}\partial_{tx}D^{s-2}\px u \phi + \int \partial_t uD^{s-2}\px uD^{s-2}\px u\phi\\
  & -2\int uD^{s-2}\Big(\px^6 u +c_1\px(u\px^3 u) + c_2\px(\px u\px^2 u) \Big)D^{s-2}\px u\phi + \\
  & -\int\px^5uD^{s-2}\px uD^{s-2}\px u \phi -c_2\int u\px^3 u D^{s-2}\px u D^{s-2}\px u \phi -c_1\int \px u \px^2 u D^{s-2}\px u\phi.
\end{align*}
This consists of terms of degree $3$ with respect to u (the number of $u$'s or its derivatives are $3$) and terms of degree $4$.
\begin{align*}
\text{Terms of degree 3} &= \int -2uD^{s-2}\px^6uD^{s-2}\px u \phi - \int \px^5 uD^{s-2}\px uD^{s-2}\px u\phi \\
 &=  \sum_{a+2b+c=7, b\geq 1} c^{a,b,c}\int \px^a u (D^{s-2}\px^b u)^2 \px^c \phi.
\end{align*}
By direct computation one see $ c^{5,1,0} =0 $. Amongst those terms, $\int \px u(D^{s-2}\px^3 u)^2 \phi $ is used to cancel out the same term from $E_1(t)$ by choosing an appropriate coefficient $ a_s$ and one can check the rest are all bounded by $O\left(\big(1+ \|J^3u\|^2_{L^1_TL^\infty_x} + \|J^su\|_{L^\infty_TL^2_x}\big)\sup_{[0,T]}\|J^su\|^2_{L^2_x}\right)$ after integrating in time.
Finally, terms of degree 4 are the following:
$$\int (c_1\px u\px^2u + c_2u \px^3 u) D^{s-2}\px u D^{s-2}\px u \phi +\int uD^{s-2}\px (c_2u\px^3 u+c_1\px u\px^2 u)D^{s-2}\px u\phi, $$
which are bounded by $O(\|\px^3\|_{L^\infty_x}\|J^s u\|^3_{L^2_x})$ by integrations by parts and Lemma~$\ref{commutator}$.
Therefore, $\eqref{localsmoothingineq}$ is obtained by integrating $ \frac{d}{dt}E_\phi $ in time.
\end{proof}

\section{Proof of Theorem~\ref{lwp} }\label{lwpproof}

Let $s > 5/2 $ and set $ \epsilon = s - 5/2$. In the subcritical case with the negative critical regularity, without loss of generality, we may assume the initial data is small. i.e. $$ \|u_0\|_{H^s} < \delta_0 $$
for $\delta_0 >0 $ small enough (to be decided later). More precisely, by the scaling invariance if $u(t,x)$ is a solution to $\eqref{fifth} $ with initial data $u_0$ , then $ u_\lambda(t,x) = \lambda^{-2}u(\frac{t}{\lambda^5},\frac{x}{\lambda})$ is a solution with initial data $u_{0,\lambda}(x) = \lambda^{-2}u_0(\frac{x}{\lambda}) $. Then, $\|\lambda^{-2}u_0(\frac{x}{\lambda})\|_{L^2} = \lambda^{-3/2}\|u_0\|_{L^2}$ and $\|\lambda^{-2}u_0(\frac{x}{\lambda})\|_{\dot{H}^s} = \lambda^{-3/2-s}\|u_0\|_{\dot{H}^s}$. So, we can downsize the initial data by choosing $\lambda $ sufficiently large. Thus, if we show that for any $u_0$ with $ \|u_0\|_{H^s} < \delta_0 $, there is a unique solution $u(t,x)\in C([0,1],H^s(\mathbb{R})) $, then for arbitrary initial data $ u_0 $, there exists a solution for $ T \gtrsim \|u_0\|_{H^s}^{-\frac{10}{3}} $.  \\ From now we set
$$ \|u_0\|_{H^s} =\delta, \qquad \qquad T=1  $$
From the previous local well-posedness result (e.g. \cite{ponce94}), for a smooth initial data we have a unique smooth solution. In view of approximating solutions, the key step is to find an \emph{a priori} bound for $ \|u\|_{L^\infty_TH^s_x} $.
\subsection{\emph{A priori bound}}
Our goal is to prove
\begin{quote}
''There exists $\delta_0 > 0$ so that if $\|u_0\|_{H^s} < \delta < \delta_0$, then $\|u\|_{L^\infty_TH^s_x}< 6\delta $.''
\\
\end{quote}
To prove this by continuity argument it suffices to show
\begin{quote}
''There exists $\delta_0 > 0$ so that if $\|u_0\|_{H^s} < \delta < \delta_0$ and $\|u\|_{L^\infty_TH^s_x} < 10\delta$,\\
     then $\|u\|_{L^\infty_TH^s_x}< 6\delta $.'' \\
\end{quote}
Without loss of generality, we may assume $10\delta_0 < \min(\frac{1}{2},\frac{1}{2a_sC}) $, where defined in \eqref{energybound}. In observance of the energy estimate $\eqref{energybound}$, one need to control $\|\px^3u\|_{L^1_TL^\infty_x}$. Note that $\|\px^3u\|_{L^1_TL^\infty_x} \leq \|\px^3u\|_{L^2_TL^\infty_x}$.\\
Applying Lemma~\ref{refinedstrichartz} (with $\alpha =1 , F = c_1\px u\px^2 + c_2u\px^3u $)

\begin{align*}
\|\px^3u\|_{L^2_TL^\infty}  \lesssim & \|D^su\|_{L^\infty_TL^2_x} + \|u\|_{L^\infty_TL^2_x} + \|D^{s-1}(\px u\px^2u)\|_{L^2_TL^2_x}
         \\
         & + \|D^{s-1}(u\px^3u)\|_{L^2_TL^2_x} + \|\px u\px^2u\|_{L^2_TL^2_x}
         + \|u\px^3u\|_{L^2_TL^2_x}.
\end{align*}
By \eqref{product}, \eqref{K-P}
\begin{align}
\|D^{s-1}(u\px^3u)\|_{L^2_TL^2_x} & \lesssim \|uD^{s-1}\px^3u\|_{L^2_TL^2_x} + \|D^{s-1}u\|_{L^\infty_TL^2_x}\|\px^3u\|_{L^2_TL^\infty_x} \nonumber\\
\|D^{s-1}(\px u\px^2u)\|_{L^2_TL^2_x} & = \|D^{s-1}\big(\frac{1}{6}\px^3(u^2)-\frac{1}{3}u\px^3u\big)\|_{L^2_TL^2_x} \nonumber\\
            &\lesssim \|D^{s-1}\px^3(u^2)\|_{L^2_TL^2_x} + \|D^{s-1}(u\px^3u)\|_{L^2_TL^2_x} \nonumber\\
     \label{intermediate}       &\lesssim \|u D^{s-1}\px^3u\|_{L^2_TL^2_x} + \|uD^{s-1}\px^3u\|_{L^2_TL^2_x} +\|D^{s-1}u\|_{L^\infty_TL^2_x}\|\px^3u\|_{L^2_TL^\infty_x},
\end{align}

\begin{align*}
\|uD^{s-1}\px^3u\|^2_{L^2_TL^2_x} &= \sum_j \|uD^{s-1}\px^3u\|^2_{L^2([0,T]\times[j,j+1])}\\
   & \leq \Big(\sum_j \|u\|^2_{L^\infty([0,T]\times[j,j+1])} \Big) \Big(\sup_j \|D^{s-1}\px^3u\|^2_{L^2([0,T]\times[j,j+1])} \Big)\\
   & =: A\cdot B.\\
\end{align*}
By Lemma~\ref{localsmoothing}
\begin{equation}
B \lesssim \big(1+ \|\px^3u\|_{L^2_TL^\infty_x} + \|J^su\|^2_{L^\infty_TL^2_x}\big)\|J^su\|^2_{L^\infty_TL^2_x}.
\end{equation}
Hence,
\begin{equation}\label{apriori10}
\|\px^3u\|_{L^2_TL^\infty_x} \lesssim \delta + \delta \|\px^3u\|_{L^2_TL^\infty} +
\delta(1+ \|\px^3u\|_{L^2_TL^\infty})\Big(\sum_j \|u\|^2_{L^\infty([0,T]\times[j,j+1])} \Big).
\end{equation}
Using the Duhamel formula and Lemma~\ref{maximal}
\begin{align*}
\Big(\sum_j &\|u\|^2_{L^\infty([0,T]\times[j,j+1])} \Big)^{1/2} \\
 &\lesssim \delta + \|\px u\px^2 u\|_{L^1_TL^2_x} + \|u\px^3u\|_{L^1_TL^2_x} + \|D^{5/4+\eta}(\px u\px^2 u)\|_{L^1_TL^2_x} + \|D^{5/4+\eta}(u\px^3u)\|_{L^1_TL^2_x} \\
 & \lesssim \delta + \|\px u\|_{L^\infty_TL^2_x}\|\px^3u\|_{L^1_TL^\infty_x} + \|D^{5/4+\eta}(\px u\px^2 u)\|_{L^2_TL^2_x} + \|D^{5/4+\eta}(u\px^3u)\|_{L^2_TL^2_x}.
\end{align*}
Since $5/4+\eta < s-1= 3/2+\epsilon $(by choosing $\eta < 1/4$ ), the repeating previous computation we have
\begin{equation*}
\|D^{5/4+\eta}(\px u\px^2u)\|_{L^2_TL^2_x} \lesssim \|u D^{5/4+\eta}\px^3u\|_{L^2_TL^2_x} + \|uD^{5/4+\eta}\px^3u\|_{L^2_TL^2_x} + \|D^{5/4+\eta}u\|_{L^\infty_TL^2_x}\|\px^3u\|_{L^2_TL^\infty_x},
\end{equation*}
\begin{align*}
\|D^{5/4+\eta}(u\px^3u)\|_{L^2_TL^2_x}  \lesssim \delta \|\px^3u\|_{L^2_TL^\infty_x} + \delta (1+ \|\px^3u\|_{L^2_TL^\infty_x}) \Big(\sum_j\|u\|^2_{L^\infty([0,T]\times[j,j+1])} \Big)^{1/2},
\end{align*}

\begin{equation}\label{apriori20}
\Big(\sum_j \|u\|^2_{L^\infty([0,T]\times[j,j+1])} \Big)^{1/2} \lesssim \delta + \delta \|\px^3u\|_{L^2_TL^\infty_x} + \delta (1+ \|\px^3u\|_{L^2_TL^\infty_x}) \Big(\sum_j\|u\|^2_{L^\infty([0,T]\times[j,j+1])} \Big)^{1/2}.
\end{equation}
Setting

$$  f(T) := \|\px^3 u\|_{L^2_TL^\infty} + \Big(\sum_j \|u\|^2_{L^\infty([0,T]\times[j,j+1])} \Big)^{1/2}, $$
by \eqref{apriori10} and \eqref{apriori20} we have
\begin{equation}
f(T) \lesssim  \delta + \delta f(T) + \delta f(T) (1+ f(T)).
\end{equation}
Using Sobolev's inequality
\begin{align*}
 f(0) &= \|\px^3 u_0\|_{L^\infty} + \Big(\sum_j \|u_0\|^2_{L^\infty([j,j+1])} \Big)^{1/2} \\
        &\lesssim  \|u_0\|_{H^s} =\delta
\end{align*}
By choosing $\delta >0$ sufficiently small, we have $ f(T) \leq C \delta $.\\
From \eqref{energybound}
$$ \|J^su\|_{L^\infty_TL^2_x} \leq 3\|u_0\|_{H^s} e^{Cf(T)} \leq 3\,\delta e^{C'\delta}.  $$
Therefore, there is $\delta_0 $ so that whenever $\delta < \delta_0$,   $$\|J^su\|_{L^\infty_TL^2_x} \leq 6\delta.$$\\

\subsection{Strong convergence}
We basically follow the $\epsilon$-approximate method of Bona-Smith \cite{BS} (see also, \cite{ponce91},\cite{KPV91})\\
We denote by $u^\epsilon(t) $ the solution of the initial value problem $\eqref{fifth} $ with initial data $ u_0^\epsilon = \rho_\epsilon * u_0$ where $ \rho_\epsilon$ is a mollifier as defined in Section 2.
From the $\emph{a priori}$ estimate we have
\begin{equation}\label{apriori}
\|J^s u^\epsilon \|_{L^\infty_TL^2_x} + \|\px^3 u^\epsilon \|_{L^2_TL^\infty_x} +  \Big(\sum_j \|u^\epsilon \|^2_{L^\infty([0,T]\times[j,j+1])} \Big)^{1/2} \\
\leq C(T,\|u_0\|_{H^s}).
\end{equation}
In order to establish the existence of a strong solution $u(t)$ as a limit in $L^\infty_TH^s_x $ one need the following:

\begin{proposition}\label{normconvergence}
For any $T>0$, $\{u^\epsilon\}_{\epsilon>0}$ converges in the $L^\infty_TH^s_x$ norm as $\epsilon $ tends to zero.
\end{proposition}

\begin{proof}
Without loss of generality, we may still assume $T=1 $ and $ \|u_0\|_{H^s}= \delta < \min(\delta_0, \delta_1)$ ,where $\delta_0 $ is found in the previous subsection and $ \delta_1$ will be decided so as to satisfy smallness conditions on the upcoming analysis.
Let $\epsilon > \epsilon' > 0 $. Denote $v=u^\epsilon, v'=u^{\epsilon'} (v_0=u_0^\epsilon, v'_0=u_0^{\epsilon'}) $ for simplicity. Define $w= v-v'$. Then $w$ satisfies the equation
\begin{equation}\label{w-eq}
\partial_t w + \px^5 w + c_1 (\px v\px^2w + \px w\px^2v') + c_2(v\px^3w + w\px^3 v') =0
\end{equation} and $\|w(0)\|_{H^s_x} \rightarrow 0$ as $\epsilon$ tends to $0$. We use the standard $o(\epsilon^k)$ notation to quantities, for which $\frac{o(\epsilon^{k})}{\epsilon^k} \rightarrow 0$ as $\epsilon \rightarrow 0$. Then, we have $\|w(0)\|_{H^s_x}=o(1)$ and want to show $\|w\|_{L^\infty_TH^s_x}=o(1)$. From now for simplicity we pretend that $\px v \px^2 w$ and $\px w\px^2 v'$ are absent (i.e. $c_1=0$). Actually, they can be handled by the same estimate used for $w \px^3v'$ and $v\px^3w$. We will explain how to handle those terms at the end of the proof. \\
First of all, we show the persistence property of the $L^2$-norm of $w$:
$$  \| w(t)\|_{L^2_x} \lesssim_{\|u_0\|_{H^s}}  \|w_0\|_{L^2}.  $$
Integrating \eqref{w-eq} against $w$
$$\frac{1}{2}\frac{d}{dt}\|w(t)\|^2_{L^2_x} \leq \|\px^3v'\|_{L^\infty_x}\|w(t)\|^2_{L^2_x} + c \int \px v \px w\px w.  $$
Here, in order to cancel out the last term we use the modified energy idea again. But in this case since $ \px^{-1}w $ may not well-defined, we use the inhomogeneous negative derivative $ J^{-1} $, which is a bounded operator in $L^2$.
Define
$$ E_w^0(t) := \|w(t)\|_{L^2_x} + a\int vJ^{-1}wJ^{-1}w.  $$
Note that $ E_w^0 \sim \|w\|_{L^2_x} $ as long as $\|v\|_{L^\infty_x}$ is sufficiently small (i.e. $\|u_0\|_{H^s} <\delta_1$). From \eqref{w-eq}
\begin{align*}
\frac{d}{dt}\int \px v J^{-1}wJ^{-1}w &=  c\int \px vJ^{-1}\px^2 wJ^{-1}\px^2 w + c\int \px^3vJ^{-1}\px wJ^{-1}\px w \\
  &+ c\int vJ^{-1}\big(v\px^3w + w\px^3v'\big) J^{-1}w.
\end{align*}
Unlike the previous case, our modification of the energy doesn't exactly cancel out the harmful term $ \int \px v\px w\px w $. But the difference is harmless if the coefficients of the first two terms match.
\begin{align*}
\Big| \int \px v \px w\px w - \px vJ^{-1}\px^2wJ^{-1}\px^2w \Big| &\leq \|\px v\|_{L^\infty} \Big(\|\px w\|^2_{L^2_x}-\|J^{-1}\px^2w\|^2_{L^2_x} \Big)\\
& = c\|\px v\|_{L^\infty} \Big( \int \xi^2\widehat{w}(\xi)^2 d\xi - \int \frac{\xi^4}{1+ \xi^2} \widehat{w}(\xi)^2 d\xi   \Big) \\
& = c\|\px v\|_{L^\infty} \Big( \int \frac{\xi^2}{1+ \xi^2} \widehat{w}(\xi)^2 d\xi   \Big) \\
& \leq c\|\px v\|_{L^\infty} \Big( \int \widehat{w}(\xi)^2 d\xi   \Big) \\
&  = c\|\px v\|_{L^\infty} \| w(t) \|^2_{L^2_x}.
\end{align*}
We choose $a$ so that the coefficients of the above two terms are opposite. \\
For the rest terms, we use $ \|J^{-1}\px\|_{L^2 \rightarrow L^2} \leq 1$.
\begin{align*}
\int v J^{-1}(\px^3v'w)J^{-1}w & \leq \|v\|_{L^\infty_x} \|J^{-1}(\px^3v'w) \|_{L^2_x} \|J^{-1}w\|_{L^2_x}\\
    & \leq  \|v\|_{L^\infty_x} \|J^{-1}(\px^3v'w) \|_{L^2_x} \|J^{-1}w\|_{L^2_x}\\
    & \leq  \|v\|_{L^\infty_x} \|\px^3v'w \|_{L^2_x} \|w\|_{L^2_x}\\
    & \leq  \|v\|_{L^\infty_x} \|\px^3v'\|_{L^\infty_x}\|w \|_{L^2_x} \|w\|_{L^2_x}.\\
\end{align*}
Using \eqref{K-P} and Sobolev embedding 
\begin{align*}
\int v J^{-1}(v\px^3w)J^{-1}w &= \int vJ^{-1}\big[\px^3(vw) - 3\px^2(\px vw) +3\px(\px^2vw) -\px^3v w  \big]J^{-1}w \\
  & \lesssim \|J^3v\|^2_{L^\infty_x} \|w\|^2_{L^2} + \int vJ^{-1}\px^2(vw)J^{-1}\px w, \\
\int vJ^{-1}\px^2(vw)J^{-1}\px w & = \int v[J^{-1}\px^2; v]\,wJ^{-1}\px w  + \int v^2J^{-1}\px^2wJ^{-1}\px w \\
   & \lesssim \|v\|_{L^\infty_x} \|\px v\|_{L^\infty} \|J^{-1}\px w \|^2_{L^2_x} - \int  v\px vJ^{-1}\px wJ^{-1}\px w \\
   & \lesssim_{\|u_0\|_{H^s}}  \|\px^3v\|_{L^\infty_x} \|w\|^2_{L^2}.
  \end{align*}
Putting them all together we have
\begin{equation}\label{w-L2}
 \|w(t)\|_{L^2_x} \lesssim \|w_0\|_{L^2_x}\exp\big(C\|u_0\|_{H^s_x}\big)
\end{equation}
as long as $\|u_0\|_{H^s} $ is small.\\
From \eqref{w_epsilon} and \eqref{w-L2} we have
\begin{equation}\label{w-epsilon-persistent}
\|w(t)\|_{L^2_x} = o(\epsilon^s).
\end{equation}
Next, we turn to $\|w(t) \|_{H^s_x}$. To use the modified energy method we define

$$ E_w(t) := \|D^sw(t)\|^2_{L^2} + \|w(t)\|_{L^2} + a_s \int v D^{s-2}\px w D^{s-2}\px w.  $$
Note that for small $\|u_0\|_{H^s} $(i.e. $<\delta_1$), we have  $$ \int v D^{s-2}\px w D^{s-2}\px w \leq \|v\|_{L^\infty_x}\|D^{s-2} w\|_{L^2} \leq c\|J^sw(t)\|_{L^2_x}, $$from which we deduce $ \|J^sw(t)\|^2_{L^2} \sim E_w(t)  $.
Our goal is to prove $$ \frac{d}{dt}E_w(t) \leq C({\|u_0\|_{H^s_x}})  E_w(t),  \qquad \qquad \text{ for  }  0\leq t \leq 1.$$
Then since $ E_w(0) =o(1) $ and $ E_w(t) \sim \|w(t)\|_{H^s_x} $, we conclude $\|w(t)\|_{H^s_x} =o(1)$ as desired.\\
From \eqref{w-eq}
\begin{equation}\label{timederivative}
\frac{1}{2}\frac{d}{dt}\|D^sw\|^2_{L^2} + c_2\int D^s(v\px^3w)D^sw + c_2\int D^s(w\px^3v')D^sw = 0.
\end{equation}
In order to estimate $ \int D^s(v\px^3w)D^sw  $, we use the commutator estimate and cancellation with $ \frac{d}{dt} \int v D^{s-2}\px w D^{s-2}\px w  $. After some integrations by parts we have
\begin{align*}
 \int D^s(v\px^3w)D^sw = &\int \Big[D^s(v\px^3w) - vD^s\px^3w - s \px v D^s\px^2 w - \frac{s(s-1)}{2} \px^2 vD^s\px w   \Big] D^sw \\
    & + c \int \px v (D^s\px w)^2   + c \int \px^3v(D^sw)^2.
\end{align*}
Here each constant $c $ has a different value.\\
On the other hand,
\begin{align*}
\frac{d}{dt}\int v D^{s-2}\px wD^{s-2}\px w & = \int v_t D^{s-2}\px wD^{s-2}\px w + 2 \int vD^{s-2}\partial_{xt}wD^{s-2}\px w \\
 = c\int \px v(D^{s-2}\px^3 w)^2 + c\int \px^3 & v (D^s w)^2 + c\int vD^{s-2}\px (v\px^3 w)D^{s-2}\px w \\
 & + c\int vD^{s-2}\px (w\px^3v)D^{s-2}\px w.
\end{align*}
We choose $a_s$ so that $ \int \px v(D^{s-2}\px^3 w)^2 $ is canceled out in $ \frac{d}{dt}E_w(t) $. By Lemma~\ref{commutator}
\begin{align*}
\int \Big[D^s(v\px^3w) - vD^s\px^3w - &s \px v D^s\px^2 w - \frac{s(s-1)}{2} \px^2 vD^s\px w   \Big] D^sw \\
  &\lesssim \|\px^3v\|_{L^\infty_x}\|w \|^2_{H^s_x} + \|\px^3w\|_{L^\infty_x}\|v \|_{H^s_x}\|w\|_{H^s_x}.
\end{align*}
For the last two terms one can use integrations by parts and the commutator estimate to show

$$\left|\int vD^{s-2}\px (v\px^3 w)D^{s-2}\px w\right| + \left|\int vD^{s-2}\px (w\px^3v)D^{s-2}\px w\right|  \lesssim O(\|u_0\|_{H^s}) \|\px^3v\|^2_{L^\infty_x}\|w\|_{H^s_x}.
$$
Hence, using \eqref{apriori}
\begin{equation}\label{second}
-c_2\int D^s(v\px^3w)D^sw + \frac{d}{dt}\int v D^{s-2}\px wD^{s-2}\px w \lesssim_s  O(\|u_0\|_{H^s}) \Big\{ \|w(t)\|^2_{H^s_x} + \|\px^3w(t)\|_{L^\infty_x}\|w(t)\|_{H^s_x} \Big\}.\\
\end{equation}
To estimate the second term in $\eqref{timederivative}$ we get some help from the global smoothing effect of the Strichartz estimate.
\begin{align}
\int &D^s(w\px^3v')D^sw \lesssim \|D^s(w\px^3 v')\|_{L^2}\|D^sw\|_{L^2}  \nonumber\\
  & \lesssim \Big( \|D^sw \cdot \px^3v'\|_{L^2} + \|w\cdot D^s\px^3v'\|_{L^2}  + \|D^sw\|_{L^2}\|\px^3v'\|_{L^\infty}   \Big) \|D^sw\|_{L^2}  \nonumber\\
  & \lesssim \Big( \|D^sw \|_{L^2}\| \px^3v'\|_{L^\infty} + \|w \|_{L^2} \| D^s\px^3v'\|_{L^\infty}  + \|D^sw\|_{L^2}\|\px^3v'\|_{L^\infty}   \Big) \|D^sw\|_{L^2}  \nonumber\\
\label{second1}  & \lesssim  \|J^sw(t)\|^2_{L^2}  + \|J^sw\|_{L^2} o(\epsilon^{s})\|J^{s+3}v'\|_{L^\infty}.
\end{align} At the last step we used \eqref{w-epsilon-persistent}.\\
Integrating in time we now claim
\begin{equation}\label{third}
o(\epsilon^s)\, \|D^{s+3}v'\|_{L^1_TL^\infty} =o(1)
\end{equation}
Lemma~$\ref{refinedstrichartz}$ (or its proof) (choosing $\alpha =1$) and Lemma~$\ref{localsmoothing}$ lead that for small $\vartheta >0 $
\begin{align*}
\|D^{s+3}v'\|&_{L^2_TL^\infty_x} \\
&\lesssim   \|D^{s+5/2+\vartheta}u_0\|_{L^2} +  \| D^{s+3/2+\vartheta }(v'\px^3v')\|_{L^2_TL^2_x}  \\
         &\lesssim \epsilon^{-5/2-\vartheta} + \|D^{s+3/2+\vartheta}v'\|_{L^\infty_TL^2_x} \|\px^3 v'\|_{L^2_TL^\infty} + \\
         &    \Big(\sum_j \|v'\|^2_{L^\infty([0,T]\times[j,j+1])} \Big)^{1/2} \Big(\sup_j \|D^{s+3/2+}\px^3v'\|_{L^2([0,T]\times[j,j+1])} \Big) \\
         & \lesssim \epsilon^{-5/2-\vartheta} +  \Big(\sum_j \|v'\|^2_{L^\infty([0,T]\times[j,j+1])} \Big)^{1/2} \|D^{s+5/2+\vartheta}v'\|_{L^\infty_TL^2_x}  \\
         & \lesssim \epsilon^{-5/2-\vartheta} + \epsilon^{-5/2-\vartheta} \Big(\sum_j \|v'\|^2_{L^\infty([0,T]\times[j,j+1])} \Big)^{1/2}\\
         & \lesssim_{\|u_0\|_{H^s}} \epsilon^{-5/2-\vartheta}
\end{align*}
where we used \eqref{v_epsilon} and \eqref{apriori}. Hence, for $s>5/2$, \eqref{third} follows.\\
From \eqref{second} and \eqref{third}
$$ \frac{d}{dt}E_w(t) \lesssim \|J^sw\|^2_{L^2_x} + (o(1) + O(\|u_0\|_{H^s})\|\px^3w\|_{L^\infty_x})\|J^sw\|_{L^2_x}. $$
Hence,
\begin{equation}\label{semifinal}
\|J^sw\|_{L^\infty_TL^2_x} \lesssim_s \Big(\|J^sw_0\|_{H^s_x} + o(1) + O(\|u_0\|_{H^s})\|\px^3w\|_{L^1_TL^\infty_x}  \Big) \exp(C\|u_0\|_{H^s}).
\end{equation}
It remains to estimate  $ \| \px^3w(t)\|_{L^1_TL^\infty_x} $.\\
For this, we argue as the proof of the \emph{a priori} bound. Set
$$
 g_T = \| \px^3w\|_{L^2_TL^\infty_x} + \Big(\sum_j \|w\|^2_{L^\infty([0,T]\times[j,j+1])} \Big)^{1/2}.
$$
Repeating the proof of Lemma~\ref{localsmoothing} and using \eqref{apriori} and \eqref{w-eq}, we get the local smoothing estimate for $ w $:

\begin{equation}\label{wlocalsmoothing}
\Big( \int^T_0\,\int_j^{j+1} |D^{s-2}\px^4w|^2 dxdt \Big) \lesssim_{\|u_0\|_{H^s_x}} \big(1+ \|\px^3w\|_{L^2_TL^\infty_x}\big)\sup_{[0,T]}\|J^sw\|^2_{L^2_x}.
\end{equation}
Using Lemma~\ref{refinedstrichartz} and \eqref{w-eq} for small $\|u_0\|_{H^s}$,
\begin{align*}
\|\px^3w\|_{L^2_TL^\infty_x} &\lesssim \|J^sw\|_{L^\infty_TL^2_x} + \|D^{s-1}(v\px^3w)\|_{L^2_TL^2_x} +\|D^{s-1}(w\px^3v')\|_{L^2_TL^2_x} \\
                  & + \|v\px^3w\|_{L^2_TL^2_x} +\|w\px^3v'\|_{L^2_TL^2_x} \\
    & \lesssim   \|J^sw\|_{L^\infty_TL^2_x} + O(\|u_0\|_{H^s})\cdot \|\px^3w\|_{L^2_TL^\infty_x} +
    \|D^{s-1}v \|_{L^\infty_TL^2_x}\|\px^3w\|_{L^2_TL^\infty_x}\\
    & + \|v \cdot D^{s-1}\px^3w\|_{L^2_TL^2_x}
    +  \|D^{s-1}w \|_{L^\infty_TL^2_x}\|\px^3v'\|_{L^2_TL^\infty_x} + \|w \cdot D^{s-1}\px^3v'\|_{L^2_TL^2_x}\\
  & \lesssim \|J^sw\|_{L^\infty_TL^2_x} + \|v \cdot D^{s-1}\px^3w\|_{L^2_TL^2_x} + \|w \cdot D^{s-1}\px^3v'\|_{L^2_TL^2_x}.
\end{align*}
Lemma~\ref{localsmoothing}, \eqref{apriori} and \eqref{wlocalsmoothing} yield
\begin{align*}
\|v \cdot D^{s-1}\px^3w\|_{L^2_TL^2_x} & \lesssim \Big(\sum_j \|v\|^2_{L^\infty([0,T]\times[j,j+1])} \Big)^{1/2}
       \Big(\sup_j \|D^{s-1}\px^3w\|_{L^2([0,T]\times[j,j+1])}   \Big)\\
       & \leq  O(\|u_0\|_{H^s_x})\,g_T,
\end{align*}

\begin{align*}
\|w \cdot D^{s-1}\px^3v'\|_{L^2_TL^2_x} &\lesssim \Big(\sum_j \|w\|^2_{L^\infty([0,T]\times[j,j+1])} \Big)^{1/2}
       \Big(\sup_j \|D^{s-1}\px^3v'\|_{L^2([0,T]\times[j,j+1])}   \Big)\\
       & \leq  O(\|u_0\|_{H^s_x})\, g_T.\\
\end{align*}
Using the maximal function estimate (Lemma~\ref{maximal})
\begin{align*}
\Big(\sum_j \|w\|^2 &_{L^\infty([0,T]\times[j,j+1])} \Big)^{1/2} \lesssim \|J^sw\|_{L^\infty_TL^2_x} + \|D^{5/4+\eta}(v\px^3w)\|_{L^2_TL^2_x} +\|D^{5/4+\eta}(w\px^3v')\|_{L^2_TL^2_x} \\
                  & + \|v\px^3w\|_{L^2_TL^2_x} +\|w\px^3v'\|_{L^2_TL^2_x} \\
    & \lesssim \|J^sw\|_{L^\infty_TL^2_x} + O(\|u_0\|_{H^s_x})\,g_T.
\end{align*}
Hence $$ g_T \lesssim_s  \|J^sw\|_{L^\infty_TL^2_x} + O(\|u_0\|_{H^s_x}) g_T. $$
By choosing $\|u_0\|_{H^s}$ sufficiently small, we have 
$$ g_T \lesssim  \big(1+O(\|u_0\|_{H^s})\big) \|w\|_{L^\infty_TH^s_x}. $$
Plugging this into \eqref{semifinal}, we obtain 
\begin{equation}\label{w-final}
\|J^sw\|_{L^\infty_TL^2_x} \lesssim_s \Big(\|J^sw_0\|_{H^s_x} + o(1) + O(\|u_0\|_{H^s}) \|J^sw\|_{L^\infty_TL^2_x}  \Big) \exp(C\|u_0\|_{H^s}).
\end{equation}
Again choosing $\|u_0\|_{H^s}$ sufficiently small and using $\|J^sw_0\|_{H^s_x} =o(1) $,
we conclude $$ \|J^sw\|_{L^\infty_TL^2_x} =o(1). $$
We finish the proof by explaining how to handle the intermediate term $\px u \px^2u $. In view of the equation \eqref{w-eq} of the difference $w$, we need to handle $ D^{s-1}(\px v\px^2w) $ and $D^{s-1}(\px w \px^2 v')$. From the above computation for $u\px^3u$ and the identity
$$ \px^3(uw) = \px^3u w + 3\px^2u \px w + 3\px u \px^2 w + u\px^3 w,$$
it suffices to estimate $ D^{s-1}\px^3(vw) $, $D^{s-1}(\px^2w\px v) $ and $D^{s-1}(\px^2w\px v')$. By \eqref{product}, \eqref{K-P} and Sobolev embedding
\begin{align*}
\|D^{s-1}\px^3(v w)\|_{L^2_x} &\lesssim \|v\cdot D^{s-1}\px^3 w\|_{L^2_x} + \|D^{s-1}\px^3 v\cdot w\|_{L^2_x} + \|D^{s-\eta}v\|_{L^r_x}\|D^{2+\eta}w\|_{L^{r'}} \\
                            \lesssim &\Big(\sum_j \|v\|^2_{L^\infty([0,T]\times[j,j+1])} \Big)^{1/2}
       \Big(\sup_j \|D^{s-1}\px^3w\|_{L^2([0,T]\times[j,j+1])}   \Big) \\
       &+\Big(\sum_j \|w\|^2_{L^\infty([0,T]\times[j,j+1])} \Big)^{1/2}
       \Big(\sup_j \|D^{s-1}\px^3v\|_{L^2([0,T]\times[j,j+1])}   \Big)\\
       & + \|J^sv\|_{L^2_x}\|J^sw\|_{L^2_x},\\
\|D^{s-1}(\px^2w\px v)\|_{L^2_x} &\lesssim \|\px^3 w\|_{L^\infty}\|D^{s-1}u\|_{L^2_x} + \|\px^2w\|_{L^\infty_x}\|D^su\|_{L^2_x} \\
                                 &\lesssim \|\px^3 w\|_{L^\infty}\|D^{s-1}u\|_{L^2_x} + \|J^sw\|_{L^2_x}\|D^su\|_{L^2_x}
\end{align*} where $ \frac{1}{2}=\frac{1}{r} +\frac{1}{r'} $, $\eta=\frac{1}{2} -\frac{1}{r}$ and $s-2-\eta \geq \frac{1}{2}-\frac{1}{r'}$.\\
We have already handled these terms in the previous analysis.

\end{proof}

\subsection{Continuous dependence}
The proof for the continuous dependence is very similar to that of Proposition~\ref{normconvergence}. We will prove that for given $ \lambda >0 $ there exists $\delta >0 $ so that if $ \|u_0 -v_0 \|_{H^s} < \delta$, then
$$ \|u-v\|_{L^\infty_TLH^s_x} < \lambda $$
where $u$ and $v$ are the solutions with initial data $u_0$ and $v_0$, respectively. \\
From Proposition~\ref{normconvergence} it follows that there exists $\epsilon_0>0$ so that for $\epsilon <\epsilon_0$
\begin{gather*}
 \| u_\epsilon - u \|_{L^\infty_TH^s_x} < \lambda/3, \\
 \| v_\epsilon - v \|_{L^\infty_TH^s_x} < \lambda/3.
\end{gather*}
We will prove there is $ \delta > 0 $ so that if $ \|u_0-v_0\| < \delta $ ,then for some $\epsilon <\epsilon_0$,
$$\label{westimate} \| u_\epsilon -v_\epsilon \|_{L^\infty_TH^s_x} < \lambda/3. $$
Define $ w = v_\epsilon -u_\epsilon $. Note $ \|w_0\|_{H^s} = \| (u_0-v_0)_\epsilon \|_{H^s} \sim \|u_0-v_0\|_{H^s_x}=\delta $ for small $\epsilon$.
The proof follows basically that of Proposition~\ref{normconvergence}. The equation for the difference is
$$\label{wepsilon-eq} \partial_t w + \px^5 w + v_\epsilon \px^3 w + w\px^3 u_\epsilon =0. $$
As before we consider the time derivative of the modified energy
$$ E_w(t) = \|D^w(t)\|^2_{L^2_x} + a_s\int v_\epsilon D^{s-2}\px w D^{s-2}\px w $$
with $E_w(0) = O(\delta) $. Following the previous argument one can easily check everything works \emph{except} one step, where we used
 \eqref{w_epsilon} and \eqref{w-epsilon-persistent} in estimating \eqref{second1}
 $$ \|w(t) \|_{L^2_x} \lesssim \|w_0\|_{L^2}  \nleq  o(\epsilon^{s}), $$
since $w_0$ is no longer of form $ f-f_\epsilon $. But since $ \|w_0\|_{L^2}\sim\|u_0-v_0\|_{L^2}$, after fixing $\epsilon >0 $ we can choose $ \delta= \delta(\epsilon) >0$ so that
$$ \|w(t) \|_{L^\infty_TL^2_x}\|D^s\px^3v_\epsilon\|_{L^2_TL^\infty_x} \lesssim  \delta \epsilon^{-s} \leq \lambda/ 10.  $$
This completes the proof of the local well-posedness for $s>\frac{5}{2}$.\\

\section{The lack of uniform continuity of the solution map}\label{illproof}

In this section we prove Theorem~\ref{ill-posed}. As mentioned before, the method used here is introduced by Koch and Tzvetkov \cite{koch-tzvetkov2}.
\subsection{Preliminaries}

We provide several preliminary results and definitions used later analysis.
Let $\phi \in C^\infty_0(\mathbb{R})$ be a bump function such that $\phi(x)=1$ for $|x|<1$ and $\phi(x)=0$ for $ |x|>2 $. Let $\phi^w \in C^\infty_0(\mathbb{R})$ be a wider bump function such that $\phi^w(x) =1 $ on the support of $\phi$. Note $\phi\,\phi^w=\phi $. For $0<\delta<1 $ and $\lambda \geq 1$, we set
$$ \phi_\lambda(x) :=\phi(\frac{x}{\lambda^{4+\delta}}) \qquad ,\phi^w_\lambda(x) :=\phi^w(\frac{x}{\lambda^{4+\delta}}). $$
In our example we exploit low frequency perturbation in high frequency wave. So, our example has a low frequency mass and a high frequency wave on the almost same support. We define the low frequency initial data
$$ u_{low}(0,x) = - \Lambda \omega \lambda^{-3}\phi^w\lambda(x), \qquad \quad \omega = \pm 1. $$
Let $u_{low}(t,x)$ be the solution to \eqref{fifth} with initial data $u_{low}(0,x) $.\\
Next, define the approximate solution
\begin{align*}
\label{apsol}
u_{ap}(t,x) & := u_{low}(t,x) - \Lambda \lambda^{-\frac{4+\delta}{2} -s}
\phi_\lambda(x) \cos ( \lambda x - \lambda^5 t - \omega t) \\
& =: u_{low}(t,x)  + u_{hi}(t,x).
\end{align*}
Note $ \|u_0\|_{H^s} = O(\Lambda) $ uniformly in $\lambda$, we can choose $ \Lambda >0 $ sufficiently small so as to satisfy all smallness conditions on $\|u_0\|_{H^s} $ required in the analysis of the local well-posedness for $s>\frac{5}{2}$. By doing so, for $ \frac{5}{2} <\sigma < s $ we guarantee the solution $u(t,x)$ exists for $ 0 \leq t \leq 1 $
and
\begin{equation}\label{u-lwp}
\| u\|_{L^\infty_TH^\sigma_x} + \|\px^3u\|_{L^2_TL^\infty} \lesssim \|u_0\|_{H^\sigma} \lesssim \lambda^{\sigma-s}.  \end{equation}
By direct computations we also have
\begin{equation}\label{uap-lwp}
 \| u_{ap}\|_{L^\infty_TH^\sigma_x} + \|\px^3u_{ap}\|_{L^2_TL^\infty} \lesssim \lambda^{-\frac{2-\delta}{2}} + \lambda^{\sigma-s} + \lambda^{\frac{2-\delta}{2}-s}.  \end{equation}
\begin{lemma}\label{pre1}
Let $s \geq 0, 0<\delta $ and $ \alpha \in \mathbb{R}$. Then,
$$  \lim_{\lambda \rightarrow \infty} \lambda^{-\frac{4+\delta}{2} -s }\|\phi_\lambda(x)\sin(\lambda x +\alpha)   \|_{H^s_x} = c \| \phi\|_{L^2_x}. $$
\end{lemma}
\begin{proof}
See \cite{koch-tzvetkov2} Lemma 2.3.
\end{proof}

Theorem~\ref{ill-posed} is a corollary of the following proposition, showing that the approximate solution constructed above is a good approximate solution in the $H^s$ sense.
\begin{proposition}\label{keyapp}
Let $ \max(0 , 2-2s) < \delta < 2$. Let $u_{\omega,\lambda}$ be the unique 
solution to the equation \eqref{fifth} with initial data
$$ u_{\omega,\lambda} (0,x) = - \omega \lambda^{-3}\phi^w_\lambda(x) -
\lambda^{- \frac{4+\delta}{2} - s}\phi_\lambda(x) \cos \lambda x  $$
and $ u_{ap}(t,x) $ be as defined above.

Then for $s>\frac{5}{2}$
\begin{equation}\label{keyapp1}
  \| u_{\omega,\lambda} - u_{ap} \|_{H^s_x}  =  o(1)
\end{equation}
holds true for $|t|<1$ as $\lambda \rightarrow \infty $. \\
Moreover, if the equation \eqref{fifth} satisfies that for all $ t$,
\begin{itemize}
\item the $L^2$ conservation law in the sense that  $\|u(t)\|_{L^2_x} =\|u(0)\|_{L^2} $ 
\item an $H^3$ conservation law in the sense that
$$  \|u(t)\|_{H^3_x}  \lesssim \|u(0)\|_{H^3_x}  \qquad \text{for small}\quad \|u\|_{H^s_x},   $$ 
\end{itemize}
then \eqref{keyapp1} holds true for $ s>0$.\\
\end{proposition}

\begin{proof}[(Proposition~\ref{keyapp} implies Theorem~\ref{ill-posed})]
By choosing $\omega = \pm 1$ we obtain two sequences of initial data
$$ u_\lambda^\pm(0,x) = \mp \lambda^{-3}\phi^w_\lambda(x) - \lambda^{-\frac{4+\delta}{2} -s}
\phi_\lambda(x) \cos ( \lambda x ). $$
Using Proposition~\ref{keyapp}
\begin{align*}
\| u_\lambda^+(t) - u_\lambda^-(t) \|_{H^s} &= \lambda^{-\frac{4+\delta}{2} -s } \| \phi_\lambda(x)[\cos(\lambda x -\lambda^5 t +t)- \cos(\lambda x -\lambda^5 t -t)  ]\|_{H^s_x} + o(1)\\
  & 2 \lambda^{-\frac{4+\delta}{2} -s } \| \phi_\lambda(x)[\sin(\lambda x -\lambda^5 t )  ]\|_{H^s_x}|\sin t| + o(1)\\
\end{align*}
where $|t|<1 $.\\
From Lemma~\ref{pre1}
$$ \lim_{\lambda \rightarrow \infty} \| u_\lambda^+(t) - u_\lambda^-(t) \|_{H^s} \geq c |\sin t|.  $$
Since $\sin t \sim t$ for $ 0\leq t \leq 1$, we complete the proof.
\end{proof}

In the following lemma, we provide several bounds for $u_{low}(t,x)$

\begin{lemma} \label{lowbound}
Let $K$ be a given positive integer.
Let $ K-2-s \geq k \geq 0 $. \\
Then the following estimates hold :
\begin{gather}
\label{low1} \| \partial^k_x u_{low}(t,\cdot)\|_{L^2(\mathbb{R})}
\lesssim_K \lambda^{-\frac{2-\delta}{2}- k(4+\delta)} \\
\label{low2}
\|\partial_x^k u_{low}(t,\cdot)\|_{L^\infty(\mathbb{R})}
\lesssim_K  \lambda^{-3-k(4+\delta)} \\
\label{low3}  \|u_{low}(t,\cdot) - u_{low}(0,\cdot) \|_{L^2(\mathbb{R})} \lesssim_K
\lambda^{-15-3\delta}
\end{gather}

\end{lemma}

\begin{proof}
We set a rescaled function
\begin{gather*}
\label{vdef} v(t,x)
:=\lambda^{2(4+\delta)}\,u_{low}(\lambda^{5(4+\delta)}t,
\lambda^{4+\delta}x),\\ 
v(0,x) = - \omega \lambda^{-3 + 2(4+\delta)}.
\phi^w(x)
\end{gather*}
Then $ v$ is a solution to
$\eqref{fifth}$ with the initial data $ v(0,x)$. We obtain
\begin{align*}
\|v(0,\cdot)\|_{H^s} &= \lambda^{5+2\delta}\| \phi^w\|_{H^s} \leq \lambda^{5+2\delta}\|
\phi^w\|_{H^K}\\
 & \lesssim_K \lambda^{5+2\delta}
\end{align*}
 and then by Theorem~\ref{lwp}
\begin{equation*}\label{ulow:t Hs}
\|v(t,\cdot)\|_{H^s} \lesssim_K \lambda^{5+2\delta},
\end{equation*}
for $ |t| \leq \min(1,c\lambda^{-\frac{10}{3}(5+2\delta)})$ and $s
> 5/2 $. Since the right hand side of $\eqref{ulow:t Hs} $ doesn't
depend on $s$, it is true for any real $ K > s > 5/2$. By Sobolev
embedding we have
\begin{equation*}
\|\partial_x^k v(t,\cdot)\|_{L^\infty} \lesssim \lambda^{5+2\delta}
\end{equation*}
for $ |t| \leq \min(1,c\lambda^{-\frac{10}{3}(5+2\delta)}) $ and $k+s < K-2 $.\\
From $\eqref{vdef} $ we can deduce \eqref{low1}, \eqref{low2} by rescaling back
$$ \partial_x^k v(t,x) =
\lambda^{(2+k)(4+\delta)}\partial_x^k
u_{low}(\lambda^{5(4+\delta)}t, \lambda^{4+\delta}x)$$

$$ \|\partial_x^ku_{low}(\lambda^{5(4+\delta)}t,
\lambda^{4+\delta}x)\|_{L^\infty_x} \lesssim \lambda^{5+2\delta
-(2+k)(4+\delta)} $$ as long as  $|t| \leq \min(1,
\lambda^{-\frac{10}{3}(5+2\delta)}).$ Hence,

$$ \|\px^k u_{low}(t,\cdot)\|_{L^\infty_x} \lesssim
\lambda^{-3-k(4+\delta)}  $$  for $|t| \leq
\lambda^{20+5\delta-\frac{10}{3}-\frac{20}{3}\delta} $ (in
particular, for $|t| \leq 1 $). \\
Similarly, from $$ \|\partial_x^ku_{low}(\lambda^{5(4+\delta)}t,
\lambda^{4+\delta}x)\|_{L^2_x} \lesssim \lambda^{5+2\delta
-(2+k)(4+\delta)} $$
we deduce  $$\|\px^k u_{low}(t,\cdot)\|_{L^2_x}
\lesssim \lambda^{-1+\frac{1}{2}\delta -k(4+\delta)}  $$ as long as
at least $ |t| \leq 1 $.\\
From $\eqref{fifth}$
\begin{align*}
\| \partial_tu_{low}(t,\cdot)\|_{L^2}   & \lesssim
\|\partial_x^5u_{low}(t,\cdot)\|_{L^2} + \|\partial_x^3 u_{low}
\|_{L^\infty}\|u_{low}\|_{L^2} +  \\ &
\|\partial_x^2 u_{low}\|_{L^\infty}\|\partial_x u\|_{L^2}\\
   & \lesssim \lambda^{-\frac{2-\delta}{2}-5(4+\delta)} +
  \lambda^{-15-3\delta-\frac{2-\delta}{2}} +
  \lambda^{-11-2\delta-\frac{2-\delta}{2}-(4+\delta)} \\
  & \lesssim \lambda^{-15-3\delta}.
\end{align*}
Integrating in time we have  $$ \|u_{low}(t,\cdot) -
u_{low}(0,\cdot) \|_{L^2} \lesssim \lambda^{-15-3\delta} $$ for $|t|
\leq 1 $.

\end{proof}
Note that in this proof $ \lesssim $ means $\lesssim_K$. Later we use Lemma~\ref{lowbound} for some bounded $k$. Once $s $ is fixed we can choose, for instance, $K > s + 100 $ and regard as $K=K(s) $.\\

\subsection{The approximate solution}
In the following lemma we show $ u_{ap} $ defined above solves \eqref{fifth} with a small error in the $H^\sigma$ sense.

\begin{lemma}
Let $s > 0 $, $0 < \delta < 2$ and $|t| \leq 1$. Set
$$
 F:= (\partial_t + \partial^5_x)u_{ap} + c_1
u_{ap}\partial_x^3u_{ap} + c_2\partial_xu_{ap}\partial_x^2u_{ap} $$
Then,
\begin{equation}\label{F-L2}
 \| F(t,\cdot) \|_{L^2} \lesssim \lambda^{-s-\delta} +
\lambda^{\frac{2-\delta}{2}-2s}  \end{equation}

Furthermore, for $\sigma > 0$,
\begin{equation}\label{F-Hs}
  \|F(t,\cdot)\|_{H^\sigma} \lesssim \lambda^{-\delta -s +\sigma } + \lambda^{\frac{2-\delta}{2}-2s + \sigma }
  \end{equation}
\end{lemma}

\begin{proof}
We decompose $F$ as follows:
\begin{align*}
F & =: F_1 + F_2 + F_3 + F_4 + F_5 + F_6 \\
F_1 & = (\partial_t + \partial_x^5)u_{low} +
c_1\partial_xu_{low}\partial_x^2u_{low} + c_2 u_{low}\partial_x^3
u_{low} \\
F_2 & = u_{hi}\partial_xu_{hi} + \partial_xu_{hi} \partial_x^2 u_{hi} \\
F_3 & = - \lambda^{-\frac{4+\delta}{2}-s}[\partial_x^5,
\phi_\lambda] \cos (\lambda x +
 \beta) \\
F_4 & = -\lambda^{-\frac{4+\delta}{2}-s}\phi_\lambda(\partial_t +
\partial_x^5 + u_{low} \partial_x^3 ) \cos (\lambda x +
 \beta) \\
F_5 & = -\lambda^{-\frac{4+\delta}{2}-s}\Bigl(
3\partial_x\phi_\lambda
\partial_x^2 \cos (\lambda x +
 \beta) + 3\partial_x^2\phi_\lambda\partial_x \cos
\Phi + \cos (\lambda x +
 \beta) \partial_x^3 \phi_\lambda \Bigr) \\
F_6 & = c_2 u_{hi}\partial_x^3 u_{low} + c_1
\partial_xu_{hi}\partial_x^2u_{low} + c_1\partial_x^2u_{hi}\partial_x
u_{low}
\end{align*}
where $ \beta = -\lambda^5t - \omega t $.\\
By definition of $\phi_\lambda(x)$ we have
$$\px^k\phi_\lambda(x)= \lambda^{-k(4+\delta)} (\px^k\phi)_\lambda(x).  $$
$ F_1 = 0 $ since $u_{low} $ is an exact solution to the
equation. In the following estimates of $F_2, F_3, F_5 $ and $ F_6$ ,
the worst term occurs when the most derivatives act on $\cos(\lambda x +
 \beta) $. So, we have the following estimates.
\begin{align*}
 F_2(x) & = \lambda^{-(4+\delta)-2s}
 \Bigl\{\phi_\lambda(x)\cos(\lambda x + \beta)\,\partial_x^3\Bigl(\phi_\lambda(x)\cos(\lambda x +
 \beta)\Bigr) \\  &  + \partial_x\Bigl(\phi_\lambda(x)\cos(\lambda x +
 \beta)\Bigr)\,\partial_x^2\Bigl(\phi_\lambda(x)\cos(\lambda x + \beta)\Bigr)\Bigr\}, \\
 \|F_2 \|_{L^2_x} & = O\Big(\lambda^{-(4+\delta)-2s + 3 +
 \frac{4+\delta}{2}}\Big)\\
 & = O(\lambda^{1-\delta/2 -2s}).
\end{align*}
\begin{align*}
F_3(x) & = - \lambda^{-\frac{4+\delta}{2}-s}[\partial_x^5,
\phi_\lambda] \cos (\lambda x +
 \beta) \\  & =
 \lambda^{-\frac{4+\delta}{2}-s}\partial_x\phi_\lambda(x)\,\partial_x^4\Bigl(\cos
 (\lambda x + \beta)\Bigr) + \text{Better terms}, \\
 \|F_3 \|_{L^2_x} & = O\Bigl(\lambda^{-\frac{4+\delta}{2}-s-4-\delta + 4 + \frac{4+\delta}{2}}
 \Bigr)\\ & = O\Bigl(\lambda^{-\delta -s}  \Bigr).
\end{align*}

\begin{align*}
F_5 & = -\lambda^{-\frac{4+\delta}{2}-s}\Bigl(
3\partial_x\phi_\lambda
\partial_x^2 \cos (\lambda x +
 \beta) + 3\partial_x^2\phi_\lambda\partial_x \cos
\Phi + \cos (\lambda x +
 \beta) \partial_x^3 \phi_\lambda \Bigr) \\
& =-\lambda^{-\frac{4+\delta}{2}-s}\Bigl[
\partial_x\phi_\lambda\,\partial_x^2\Bigl(\cos (\lambda x +
\beta)\Bigr) + \text{Better terms} \Bigr], \\
  \|F_5 \|_{L^2_x} & = O\Bigl(\lambda^{-\frac{4+\delta}{2}-s -4-\delta +2 + \frac{4+\delta}{2}}
  \Bigr)\\ & = O\Bigl(\lambda^{-2-\delta -s}\Bigr).
\end{align*}

\begin{align*}
F_6 & = c_2 u_{hi}\partial_x^3 u_{low} + c_1
\partial_xu_{hi}\partial_x^2u_{low} + c_1\partial_x^2u_{hi}\partial_x
u_{low}\\ & = \phi_\lambda \partial_x^2\bigl(\cos(\lambda x + \beta)
\bigr)\,\partial_xu_{low} + \text{Better terms}, \\ 
\| F_6 \|_{L^2_x} & = O\Bigl(\lambda^{-\frac{4+\delta}{2}-s +2 -3-4-\delta
+ \frac{4+\delta}{2}} \Bigr) \\ & = O\Bigl(\lambda^{-s -5-\delta}
\Bigr).
\end{align*}
Now it remains to estimate $F_4$. For this we use \eqref{low3}.
$$
F_4 = -\lambda^{-\frac{4+\delta}{2}-s}\phi_\lambda(\partial_t +
\partial_x^5 + u_{low} \partial_x^3 ) \cos (\lambda x +
 \beta). $$
A direct computation yield
\begin{align*}  \bigl[\partial_t + \partial_x^5\bigr] \cos (\lambda x - \lambda^5 t -
\omega
t) &= \sin (\lambda x - \lambda^5 t - \omega t)\cdot \omega \\
\partial_x^3 \cos (\lambda x - \lambda^5 t - \omega
t) &= \lambda^3 \sin(\lambda x - \lambda^5 t - \omega t).
\end{align*}
We use the facts that $\phi_\lambda\,\phi^w_\lambda =
\phi_\lambda $ and $u_{low}(0,x) = - \lambda^3\,\omega
\phi^w_\lambda(x) $ to get
\begin{align*}
 F_4(x) & = -\lambda^{-\frac{4+\delta}{2}-s}\phi_\lambda(x)\bigl[
u_{low}(t,x)\lambda^3 - \omega \bigr] \sin (\lambda x -\lambda^5 t -
\omega t) \\
& = -\lambda^{-\frac{4+\delta}{2}-s}\phi_\lambda(x)\bigl[
u_{low}(t,x)\lambda^3 - u_{low}(0,x)\lambda^3 \bigr] \sin (\lambda x
-\lambda^5 t - \omega t) \\
\|\,F_4\,\|_{L^2_x} & \lesssim \lambda^{-\frac{4+\delta}{2}-s+3}\lambda^{-15-3\delta}\lambda^{\frac{4+\delta}{2}}
\\ & = O\Bigl(\lambda^{-12-s-3\delta}\Bigr)
\end{align*}
where we used $\eqref{low3}$. \\
Next, we analyze $ \|D^\sigma F\|_{L^2_x} $. For any Schwartz function $f$, we have
$$D^\sigma f(\lambda^p x) = \lambda^{p\sigma}(D^\sigma f)(\lambda^p x). $$
Each $ F_i $ for $i=1,\cdots,6 $ is a product of a low frequency function and a high frequency function.\\
Let $f_\lambda(x)=f(\frac{x}{\lambda^a})$ and $g_\lambda(x)= g(\lambda^bx)$ for some $a,b >0$.
Then, in general by \eqref{product}
\begin{align*}
 \|D^\sigma (f_\lambda g_\lambda ) \|_{L^2} & \lesssim  \|f_\lambda \cdot D^\sigma g_\lambda\|_{L^2} + \|D^\sigma f_\lambda \cdot g_\lambda\|_{L^2} + \|D^\sigma f_\lambda\|_{L^2}\| g_\lambda\|_{L^\infty} \\
 & \lesssim \lambda^{\sigma b} \Big(\|\|f_\lambda (D^\sigma g)_\lambda \|_{L^2} + \|(D^\sigma f)_\lambda\cdot g_\lambda\|_{L^2} +\|(D^\sigma f)_\lambda\|_{L^2}\| g_\lambda\|_{L^\infty} \Big).
\end{align*}
Thus, by checking above computations we have
$$ \|D^\sigma F_i\|_{L^2} \lesssim_\sigma \lambda^\sigma \|F_i\|_{L^2}  $$
which implies \eqref{F-Hs}.
\end{proof}

\subsection{Proof of Proposition~\ref{keyapp}}

We set the difference between the genuine solution and the approximate
solution
$$ w_{\omega, \lambda}:= u_{\omega, \lambda} - u_{ap} $$ where
$u_{\omega, \lambda} $ is the solution to $\eqref{fifth}$.\\
Our goal is to prove
$$ \|w_{\omega, \lambda} \|_{H^s} =
  o(1),  \qquad \qquad \text{ as  }  \lambda \rightarrow \infty.  $$
  First, we do the case $ s>\frac{5}{2} $, where the local well-posedness theory is available. From now we denote $ w:=w_{\omega,\lambda} $ and $u:=u_{\omega,\lambda}$ for simplicity.\\
From \eqref{fifth} we have an equation for $w$
\begin{equation}\label{w-eq-ill}
\partial_tw + \px^5 w + c_1( \px w\px^2u_{ap} + \px u\px^2w) + c_2(w\px^3u_{ap} + u\px^3w) + F = 0. 
\end{equation}
We find the $L^2$ persistence property first. Arguing as in \eqref{w-L2} (using a correctional term), we have
$$
\|w(t)\|^2_{L^2_x} \lesssim \Big(\|\px^3u\|_{L^\infty_x} + \|\px^3u_{ap}\|_{L^\infty_x} \Big)\|w(t)\|^2_{L^2} + \|F(t)\|_{L^2_x}\|w\|_{L^2_x}.
$$
Since $w(0)=0$, from \eqref{u-lwp}, \eqref{uap-lwp} and \eqref{F-L2}
\begin{equation}\label{w-L2-ill}
  \|w\|_{L^\infty_TL^2_x} \lesssim \|F\|_{L^1_TL^2_x} =O(\lambda^{-s-\beta})
\end{equation}
for $\beta= \min(\delta, -\frac{2-\delta}{2}+s) > 0 $. \\
For the $H^s$ persistence, we argue as in Section~\ref{lwpproof} and use the $L^2$ persistence.
From \eqref{second}, \eqref{w-final}
\begin{align}
\|J^sw\|_{L^\infty_TL^2_x} &\lesssim
\Big\{ \|w_0\|_{H^s} + \|w\|_{L^\infty_TL^2_x}\|D^{s+3}u_{ap}\|_{L^1_TL^\infty_x}  \nonumber \\
\label{w-Hs-ill}  &+ \big(\|F\|_{L^1_TH^s_x} + \|\px^3u\|_{L^1_TL^\infty_x} + \|\px^3u_{ap}\|_{L^1_TL^\infty_x} \big) \|J^sw\|_{L^\infty_TL^2_x}   \Big\} \exp(C\|u_0\|_{H^s}).
\end{align}
A direct computation yields
\begin{equation}\label{w-L2-ill2}
 \|D^{s+3}u_{ap}(t)\|_{L^\infty_x}  \lesssim \lambda^{1-\delta/2} \leq \lambda^s.
\end{equation}
Using $w(0)=0$, \eqref{u-lwp}, \eqref{uap-lwp}, \eqref{F-Hs}, \eqref{w-L2-ill}, \eqref{w-L2-ill2}, in \eqref{w-Hs-ill} we conclude
$$  \|J^sw\|_{L^\infty_TL^2_x} \lesssim o(1) $$
as $\lambda \rightarrow \infty $. \\

Next, in the case $0< s \leq \frac{5}{2}$ we use interpolation between regularity exponents. To get the $L^2$ persistence of $w$ we use the $L^2$ conservation law. From \eqref{w-eq-ill} and $ u=w+u_{ap} $ we obtain
$$\partial_tw + \px^5 w + c_1( \px w\px^2u_{ap} + \px w\px^2w + \px u_{ap}\px^2w) + c_2(w\px^3u_{ap} + w\px^3w + u_{ap}\px^3w) + F = 0. $$
Note that the $L^2$ conservation law yields $ \int c_1\px w\px^2w w + c_2 w\px^3w w =0 $. Again, arguing as in \eqref{w-L2} we have
$$\frac{d}{dt}\|w(t)\|_{L^2_x} \lesssim \|\px^3u_{ap}(t)\|_{L^\infty_x}\|w(t)\|_{L^2_x} + \|F(t)\|_{L^2_x}. $$  
From \eqref{F-L2} and \eqref{uap-lwp}
\begin{equation}\label{w-L2-ill2}
 \|w(t)\|_{L^2_x} \lesssim \lambda^{-s-\beta}. 
\end{equation} 
On the other hand, from the $H^3 $ conservation law, we have
\begin{equation*}
  \|u(t)\|_{H^3_x} \lesssim \|u(0)\|_{H^3} \lesssim \lambda^{3-s}.
\end{equation*}
By a direct computation we also have
$$ \|u_{ap}(t)\|_{H^s_x} \lesssim \lambda^{3-s}. $$
Combining these estimates together we obtain
\begin{equation}\label{H3-conservation}
  \|w(t)\|_{H^3_x} \lesssim  \lambda^{3-s}.
\end{equation}
An interpolation between \eqref{H3-conservation} and \eqref{w-L2-ill2} shows
\begin{align*}
\|w(t)\|_{H^s_x} &\lesssim \|w(t)\|_{H^3_x}^{\frac{s}{3}} \|w(t)\|^{\frac{3-s}{3}}_{L^2_x}\\
   & \lesssim \lambda^{\frac{(3-s)s}{3}}\cdot \lambda^{ \frac{3-s}{3}(-s-\beta)} \\
   & = \lambda^{-\beta\frac{3-s}{3}}
\end{align*}
which completes the proof.

\end{document}